\documentclass{article}
\title{ A Nonconforming Finite Element Approximation for the von Karman Equations}

\author{Gouranga Mallik\footnote{Department of Mathematics, Indian Institute of Technology Bombay, Powai, Mumbai 400076, India. Email. gouranga@math.iitb.ac.in}
 \; and  Neela Nataraj\footnote{Department of Mathematics, Indian Institute of Technology Bombay, Powai, Mumbai 400076, India. Email. neela@math.iitb.ac.in}}

\usepackage{amsmath,amsthm,amssymb}
\usepackage{gensymb}
\usepackage[square,sort&compress,comma,numbers]{natbib}
\usepackage{textcase}
\usepackage{subfig}
\usepackage{graphicx}
\usepackage[margin=3cm]{geometry}
\usepackage{tikz}
\usepackage{caption}
\usetikzlibrary{shapes,calc}
\usepackage{verbatim}
\chardef\bslash=`\\ 

\newtheorem{thm}{Theorem}[section]

\newtheorem{lem}[thm]{Lemma}

\theoremstyle{definition}
\newtheorem{defn}{Definition}[section]

\theoremstyle{remark}
\newtheorem{rem}{Remark}[section]

\numberwithin{equation}{section}
\newcommand{\bB}{\mathbb B}

\newcommand{\bR}{\mathbb R}

\newcommand{\fC}{\mathfrak C}
\newcommand{\cA}{\mathcal A}

\newcommand{\cV}{\mathcal V}
\newcommand{\cT}{\mathcal T}

\newcommand{\map}{\longrightarrow}

\newcommand{\lt}{L^2(\Omega)}
\newcommand{\ltsq}{(L^2(\Omega))^2}
\newcommand{\hto}{H^2_0(\Omega)}

\newcommand{\integ}{\int_\Omega}
\newcommand{\sit}{\sum_{T\in\mathcal{T}_h}\int_T}
\newcommand{\sidt}{\sum_{T\in\mathcal{T}_h}\int_{\partial T}}
\newcommand{\dx}{\;\textit{dx}}
\newcommand{\ds}{\;\textit{ds}}

\newcommand{\sumt}{\sum_{T\in\mathcal{T}_h}}

\newcommand{\bphi}{\bar\Phi}

\newcommand{\fl}{\quad\forall\,}

\newcommand{\half}{\frac{1}{2}}
\newcommand{\trinl}{\ensuremath{\left| \! \left| \! \left|}}
\newcommand{\trinr}{\ensuremath{\right| \! \right| \! \right|}}

\newcommand{\Holder}{H\"{o}lder's }
\newcommand{\cof}{{\rm cof}}
\numberwithin{equation}{section}
\date{}
\begin{document}
\maketitle

\begin{abstract} 
 In this paper, a nonconforming finite element method has been proposed and analyzed for the von K\'{a}rm\'{a}n equations that describe bending of thin elastic plates. Optimal order error estimates in broken energy and $H^1$ norms are derived under minimal regularity assumptions.  Numerical results that justify the theoretical results are presented.
\end{abstract}

\hspace*{1cm}{\bf Keywords.} von K\'{a}rm\'{a}n equations, Morley element, plate bending, non-linear
\medskip\\
\hspace*{2cm}{\bf AMS subject classifications. } 35J61, 65N12, 65N30

\section{Introduction}\label{intro}

Let $\Omega \subset\mathbb{R}^2$ be a polygonal domain  with boundary $\partial \Omega$. Consider the von K\'{a}rm\'{a}n equations
   for the deflection of very thin elastic plates that are modeled  by a  non-linear system of fourth-order partial differential equations with two unknown functions defined by:  for given $f\in L^{2}(\Omega)$, seek the vertical displacement $u$ and the Airy stress function $v$ 
  such that
 \begin{equation}\label{vke}
   \left.
    \begin{array}{l l}
      \Delta^2 u &=[u,v]+f \\
      \Delta^2 v &=-\half[u,u]
    \end{array}
   \right\} \text{in } \Omega
  \end{equation}
  with clamped boundary conditions
  \begin{equation}\label{vkb}
   u=\frac{\partial u}{\partial \nu} = v = \frac{\partial v}{\partial \nu} = 0 \text{  on  } \partial\Omega,
  \end{equation}
  where the biharmonic operator $\Delta^2$ and the von K\'{a}rm\'{a}n bracket $[\cdot,\cdot]$ are defined by 
  $$\Delta^2\varphi:=\varphi_{xxxx}+2\varphi_{xxyy}+\varphi_{yyyy} \mbox{ and }
  [\eta,\chi]:=\eta_{xx}\chi_{yy}+\eta_{yy}\chi_{xx}-2\eta_{xy}\chi_{xy}=\cof(D^2\eta):D^2\chi,$$
  $\cof(D^2\eta)$ denotes the co-factor matrix of $D^2\eta$ and $\nu$ denotes the unit outward normal to the boundary $\partial\Omega$ of $\Omega$.
  
\medskip

Depending on the thickness to length ratio, several plate models have been studied in literature; the most important ones being linear models like Kirchhoff and Reissner-Mindlin plates for {\it  thin} and {\it moderately thick} plates respectively; and non-linear von K\'{a}rm\'{a}n plate model  for {\it very thin} plates. Many practical applications deal with the Kirchhoff model for {\it thin} plates in which the transverse shear deformation is negligible. 
On the other hand, the Reissner-Mindlin plate model for {\it moderately thick} plates takes into consideration the shear deformation. The displacements of {\it very thin} plates are so large that a non-linear model is essential to consider the membrane action. The assumptions made in the von K\'{a}rm\'{a}n model are similar to those of Kirchhoff model except for the linearization of the strain tensor, which in fact, leads to the non-linearity in the model. 
\medskip

For the theoretical study as regards the existence of solutions, regularity and bifurcation phenomena of  von K\'{a}rm\'{a}n equations, see \cite{CiarletPlates, Knightly, Fife, Berger,BergerFife, BlumRannacher} and the references therein. Due to the importance of the problem in application areas, several numerical approaches have also been attempted in the past. The major challenges are posed by the non-linearity and the higher order nature of the equations. The convergence analysis and error bounds for conforming finite element methods are analyzed in \cite{Brezzi}. The papers \cite{Miyoshi, Reinhart} and \cite{Quarteroni} investigate and analyze the Hellan-Hermann-Miyoshi mixed finite element method and a stress-hybrid method, respectively for the von K\'{a}rm\'{a}n equations. In these papers, the authors simultaneously approximate the unknown functions and their derivatives.  The papers \cite{Brezzi, Miyoshi, Quarteroni} deal with the approximation and error bounds for isolated solutions, thereby not discussing the difficulties arising from the non-uniqueness of the solution and the bifurcation phenomena.
\medskip

Over the last few decades, the finite element methodology has developed in various directions. For higher-order problems, nonconforming methods and discontinuous Galerkin methods are gaining popularity as they have a clear advantage over conforming finite elements with respect to simplicity in implementation. 
In this paper, an attempt has been made to study the von K\'{a}rm\'{a}n equations using nonconforming Morley finite elements. The Morley finite element method has been proposed and analyzed for the biharmonic equation  in \cite{MingXu} and for the Monge-Amp\`{e}re equation in \cite{MNeilan}.  In \cite{XLR}, a two level additive Schwarz method for a non-linear biharmonic equation using Morley elements is discussed under the assumption of smallness of data. The $C^0$ interior penalty method, a variant of the discontinuous Galerkin method has been used to analyze the Monge-Amp\`{e}re equation in \cite{Gudi}.
\medskip

The solutions $u,v$ of clamped von K\'{a}rm\'{a}n equations defined on a polygonal domain belong to $H^2_0(\Omega)\cap H^{2+\alpha}(\Omega)$\cite{BlumRannacher}, where $\alpha\in (\half, 1]$ referred to as the index of elliptic regularity is determined by the interior angles of $\Omega$. Note that when $\Omega$ is convex, $\alpha=1$. This paper discusses a nonconforming finite element discretization of \eqref{vke}-\eqref{vkb} and develops $a~priori$ error estimates for the displacement and Airy stress functions in polygonal domains with possible corner singularities. To highlight the contributions of this work, we have 
\begin{itemize}
 \item obtained an approximation of an isolated solution pair $(u,v)$ of \eqref{vke}-\eqref{vkb} using nonconforming Morley elements; 
 \item developed optimal order error estimates in broken energy and $H^1$ norms under realistic regularity assumptions;
 \item performed numerical experiments that justify the theoretical results.
\end{itemize}
 The advantages of the method are that the nonconforming Morley elements which are based on piecewise quadratic polynomials are simpler to use and have lesser number of degrees of freedom in comparison with the conforming Argyris finite elements with 21 degrees of freedom in a triangle or the Bogner-Fox-Schmit finite elements with 16 degrees of freedom in a rectangle. Moreover, the method is easier to implement than mixed/hybrid finite element methods. 

\medskip
The difficulties due to non-conformity of the space increases the technicalities in the proofs of error estimates. Moreover, one loses the symmetry property with respect to all the variables in the discrete formulation for nonconforming case. An important aid in the proofs is  a companion conforming operator, also known in the literature as the enriching operator which maps the elements in the nonconforming finite element space to that of the conforming space. Also, as proved in \cite{HuShi} for the biharmonic problem, it is true that  when Morley finite elements are used for the von K\'{a}rm\'{a}n equations, the $L^2$ error estimates cannot be further improved. This is evident from the results of the numerical experiments presented in Section ~\ref{sec:num}. 

\medskip
The paper is organized as follows.  Section ~\ref{intro} is introductory and Section ~\ref{sec:weakformulation} introduces the weak formulation for the problem. This is followed by description of nonconforming finite element formulation in Section ~\ref{sec:ncfem}. Section ~\ref{sec:ee} deals with the existence of the discrete solution and the error estimates in broken energy and $H^1$ norms. The results of the numerical experiments are presented in Section ~\ref{sec:num}. Conclusions and perspectives are discussed in Section ~\ref{conclusions}. The analysis of a more generalized form of \eqref{vke}-\eqref{vkb} is dealt with  in Appendix A.

\medskip

Throughout the paper, standard notations on Lebesgue and Sobolev spaces and their norms are employed. We denote the standard $L^2$ scalar or vector inner product by $(\cdot,\cdot)$ and the standard norm on $H^{s}(\Omega)$, for $s>0$ by  $\|\cdot\|_{s}$. 
The positive constants $C$ appearing in the inequalities denote generic constants which may depend on the domain $\Omega$ but not on the mesh-size.

\section{Weak formulation}
\label{sec:weakformulation}

The weak formulation corresponding to \eqref{vke}-\eqref{vkb} is: given $f\in\lt$, find $u,v\in \: V:=\hto$ such that
\begin{subequations}\label{wform}
  \begin{align}
   & a(u,\varphi_1)+ b(u,v,\varphi_1)+b(v,u,\varphi_1)=l(\varphi_1)   \fl\varphi_1\in V\label{wforma}\\
   & a(v,\varphi_2)-b(u,u,\varphi_2)   =0            \fl\varphi_2 \in V\label{wformb}
  \end{align}
\end{subequations}
where $\fl\eta,\chi,\varphi\in V$, 
  \begin{align*}
   &a(\eta,\chi):=\integ D^2 \eta:D^2\chi\dx, \; \;  b(\eta,\chi,\varphi):=\half\integ \cof(D^2\eta)D\chi\cdot D\varphi\dx \; \;  {\mbox {and}}      \; \;   
   l(\varphi):=(f,\varphi).
  \end{align*}
Note that $b(\cdot,\cdot,\cdot)$ is derived using the divergence-free rows property~\cite{Evans,MNeilan}. Since the Hessian matrix $D^2\eta$ is symmetric, $\cof(D^2\eta)$ is symmetric. Consequently, $b(\cdot,\cdot,\cdot)$ is symmetric with respect to the  second and third variables, that is,   $b(\eta,\xi,\varphi)=b(\eta,\varphi,\xi)$. Moreover,  since $[\cdot,\cdot]$ is symmetric, $b(\cdot,\cdot,\cdot)$ is symmetric with respect to all the variables in the weak formulation. 

\medskip
An equivalent vector form of the weak formulation which will be also used in the analysis is defined as: for $F=(f,0)$ with $f\in\lt$, seek $\Psi=(u,v)\in \cV:=V\times V$ such that
   \begin{equation}\label{vform}
      A(\Psi,\Phi)+B(\Psi,\Psi,\Phi)=L(\Phi) \fl \Phi \in \cV
   \end{equation}
   where$\fl\, \Xi=(\xi_{1},\xi_{2}),\Theta=(\theta_{1},\theta_{2})$ and $ \Phi=(\varphi_{1},\varphi_{2})\in \mathcal{V}$,
   \begin{align}
    &A(\Theta,\Phi):=a(\theta_1,\varphi_1)+a(\theta_2,\varphi_2),\label{defnA}\\
    &B(\Xi,\Theta,\Phi):=b(\xi_{1},\theta_{2},\varphi_{1})+b(\xi_{2},\theta_{1},\varphi_{1})-b(\xi_{1},\theta_{1},\varphi_{2}) \; \; \mbox{and}\label{defnB} \\
    &L(\Phi):=(f,\varphi_1).\label{defnL}
   \end{align} 

It is easy to verify that the bilinear forms $A(\cdot,\cdot)$ and $B(\cdot,\cdot,\cdot)$ satisfy the following continuity and coercivity properties. That is, there exist constants $C$ such that
\begin{eqnarray}
 {A}(\Theta,\Phi)&\leq& C\trinl\Theta\trinr_2 \: \trinl\Phi\trinr_2  \quad \forall \Theta,\,\Phi\in \mathcal{V},\\
 {A}(\Theta,\Theta)& \geq&  C\trinl\Theta\trinr_2^2 \quad \forall \Theta \in  \mathcal{V}, \\
B(\Xi, \Theta, \Phi) & \leq & C \trinl\Xi\trinr_2 \: \trinl\Theta\trinr_2 \: \trinl\Phi\trinr_2 \quad \forall \Xi, \,\Theta, \, \Phi \in  \mathcal{V},
\end{eqnarray}
where the product norm $\trinl\Phi\trinr_2:=\sqrt{A(\Phi,\Phi)}\fl\Phi\in\cV$. In the sequel, the product norm defined on $(H^s(\Omega))^2$ and $(L^2(\Omega))^2$ are denoted by $\trinl\cdot\trinr_s$ and $\trinl\cdot\trinr$, respectively.

For the results on existence of solution of the weak formulation, we refer to \cite{Berger,BergerBook, Knightly, CiarletPlates}. More precisely, the weak solution $\Psi=(u,v)$ of \eqref{vke}-\eqref{vkb} can be characterized as the solution of the operator equation $I\Psi=T\Psi$ defined on $\cV$ where $T$ is a compact operator on $\cV$ and $I$ is an identity operator on $\cV$. In \cite{Knightly}, it has been proved that there exists at least one solution of the operator equation. Also, the uniqueness of solution under the assumption on smallness of the data function $f$ has been derived.

In this paper, we follow \cite{Brezzi} and assume that the solution $\Psi=(u,v)$ is isolated. That is, the linearized problem defined by:  for given $G=(g_1,g_2)\in( L^2(\Omega) )^2 \subset {\mathcal V}'$, seek $\Theta=(\theta_1,\theta_2)\in \mathcal{V}$ such that
   \begin{equation}\label{vforml}
     \cA(\Theta,\Phi)=(G,\Phi) \fl \Phi \in \mathcal{V}
   \end{equation}
where $\cA(\Theta,\Phi):=A(\Theta,\Phi)+B(\Psi,\Theta,\Phi)+B(\Theta,\Psi,\Phi)$ is well posed and satisfies the $a~priori$ bounds
\begin{equation}\label{apriorilin23}
\trinl\boldsymbol{\Theta}\trinr_2\leq C\trinl G\trinr, \quad \trinl\boldsymbol{\Theta}\trinr_{2+\alpha}\leq C\trinl G\trinr
\end{equation}
where $\alpha$ is the index of elliptic regularity.

\section{Nonconforming Finite Element Method (NCFEM)}   
\label{sec:ncfem}

In the first subsection, the Morley element is defined and some preliminaries are introduced. In the second subsection, nonconforming finite element formulation for von K\'{a}rm\'{a}n equations and the corresponding linearized problem are presented. Some properties and auxiliary results necessary for the analysis are discussed in the third subsection.

\subsection{The Morley Element}
Let $\mathcal T_h$ be a regular, quasi-uniform triangulation \cite{Brenner,Ciarlet} of $\bar\Omega$ into closed triangles. Set $h_T={\rm diam}(T)\fl T\in \mathcal{T}_h$ and $h=\max_{T\in\mathcal{T}_h}h_T$. For $T\in \mathcal{T}_h$ with vertices $a_i=(x_i,y_i),\: i=1,2,3$, let $m_4, m_5$ and $m_6$ denote the midpoints of the edges opposite to the vertices $a_1, a_2$ and $ a_3$ respectively (see Figure~\ref{fig:MorleyElement}). 
We denote the set of vertices (resp. edges) of $\mathcal{T}_h$ by $\mathfrak{V}_h$ (resp.  $\mathfrak{E}_h$). 
 For $ e\in\mathfrak{E}_h$, let $h_e={\rm diam}(e)$.

\begin{figure}
\begin{center}
\begin{tikzpicture}
\node[regular polygon, regular polygon sides=3, draw, minimum size=4cm]
(m) at (0,0) {};

\fill [black] (m.corner 1) circle (2pt);

\put(-62,-31){$a_1$}
\put(52,-31){$a_2$}
\put(-7,62){$a_3$}

\fill [black] (m.corner 2) circle (2pt);

\fill [black] (m.corner 3) circle (2pt);

\draw [-latex, thick] (m.side 1) -- ($(m.side 1)!0.5!90:(m.corner 1)$);
\put(8,8){$m_4$}
\draw [-latex, thick] (m.side 2) -- ($(m.side 2)!0.5!90:(m.corner 2)$);
\put(-22,8){$m_5$}
\draw [-latex, thick] (m.side 3) -- ($(m.side 3)!0.5!90:(m.corner 3)$);
\put(-7,-25){$m_6$}
\end{tikzpicture}
\caption{Morley element}\label{fig:MorleyElement}
\end{center}

\end{figure}
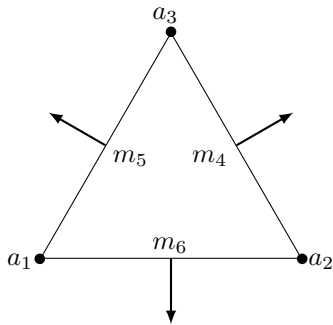

     
  \begin{defn}\cite{ Ciarlet}
  The {\it Morley finite element} is a triplet $(T,P_T,\Phi_T)$ where
  \begin{itemize}
   \item $T$ is a triangle
   \item $P_T=P_2(T)$ is the space of all quadratic polynomials on $T$ and
   \item $\Phi_T=\{\phi_i\}_{i=1}^6$ are the degrees of freedom defined by:
\vspace{-0.1in}
   $$\phi_i(v)=v(a_i),\; i=1,2,3\text{ and }\phi_i(v)=\frac{\partial v}{\partial\nu}(m_i),\; i=4,5,6.$$
  \end{itemize} 
\end{defn}
\smallskip
 
The nonconforming {\it Morley element space} associated with the triangulation $\mathcal{T}_h$ is defined by
  \begin{align*}
      V_h:=\Big\{&\varphi \in \lt\, :\, \varphi|_{T}\in P_2(T)\fl T\in \mathcal{T}_h,\,
      \varphi\text{ is continuous at the vertices $\{ a_i\}_{i=1}^3$ of the triangle }\\
      &\text{and the normal derivatives of } \varphi\text{ at the midpoint of the edges $\{ m_i\}_{i=4}^6$ are continuous, }\\
      &\varphi =0\text{ at the vertices on }\partial\Omega,\;
      \frac{\partial \varphi }{\partial \nu} =0\text{ at the midpoint of the edges on }\partial\Omega\Big\}.
    \end{align*}

For $\varphi \in V_h$ and $\Phi=(\varphi_1,\varphi_2) \in \cV_h:=V_h\times V_h$, the mesh dependent semi-norms which are equivalent to the norms  denoted as  $|\varphi|_{2,h}$ and
$\trinl\Phi\trinr_{2,h}$, respectively, are defined by:
$$ 
 |\varphi|_{2,h}^2:=\sum_{T\in\mathcal{T}_h}|\varphi|_{2,T}^2, \quad
\trinl\Phi\trinr_{2,h}^2:=|\varphi_{1}|_{2,h}^2+|\varphi_{2}|_{2,h}^2.
$$
Also, for a non-negative integer $m,\:1\leq p<\infty$ and $\varphi\in W^{m,p}(\Omega;\cT_h)$  
$$ 
 |\varphi|_{m,p,h}^2:=\sum_{T\in\mathcal{T}_h}|\varphi|_{m,p,T}^2,\quad
 \|\varphi\|_{m,p,h}^2:=\sum_{T\in\mathcal{T}_h}\|\varphi\|_{m,p,T}^2,
$$
and for $p=\infty$
\begin{equation*}
 |\varphi|_{m,\infty,h}:=\max_{T\in\mathcal{T}_h}|\varphi|_{m,\infty,T},\quad
 \|\varphi\|_{m,\infty,h}:=\max_{T\in\mathcal{T}_h}\|\varphi\|_{m,\infty,T},
\end{equation*}
where $|\cdot|_{m,p,T}$ and $\|\cdot\|_{m,p,T}$ denote the usual semi-norm and norm in the Banach space $W^{m,p}(T)$ and $W^{m,p}(\Omega;\cT_h)$ denotes the  broken Sobolev space with respect to the mesh $\cT_h$.  For $\Phi=(\varphi_{1},\varphi_{2})$ with $\varphi_1,\varphi_2\in W^{m,p}(\Omega;\cT_h)$, define
$\displaystyle
     \trinl\Phi\trinr_{m,p,h}^2:=|\varphi_{1}|_{m,p,h}^2+|\varphi_{2}|_{m,p,h}^2.\ $
When $p=2$, the notation is abbreviated as $|\cdot|_{m,h}$ and $\|\cdot\|_{m,h}$.


\subsection{Nonconforming Finite Element Formulation}
 The NCFEM formulation corresponding to \eqref{wforma}-\eqref{wformb} can be stated as: for $f\in\lt$, seek $(u_h,v_h)\in \mathcal{V}_h$ such that
 \begin{subequations}\label{wformd}
   \begin{align}
    &a_h(u_h,\varphi_1)+b_h(u_h,v_h,\varphi_1)+b_h(v_h,u_h,\varphi_1)=l_h(\varphi_1)   \fl\varphi_1\in V_h \label{wformda}\\
    & a_h(v_h,\varphi_2)-b_h(u_h,u_h,\varphi_2) =0               \fl\varphi_2 \in V_h \label{wformdb}
   \end{align}
 \end{subequations}
 where $\fl\eta,\chi,\varphi\in V_h$, 
   \begin{align*}
   &  a_h(\eta,\chi):=\sit D^2 \eta:D^2\chi\dx, \; \; b_h(\eta,\chi,\varphi):=\half\sit \cof(D^2\eta)D\chi\cdot D\varphi\dx \; \mbox{ and }  \\
& l_h(\varphi):=\sit f\varphi\dx.
   \end{align*}
As in the continuous formulation,  the discrete form $b_h(\cdot,\cdot,\cdot)$ is symmetric with respect to the second and third variables. However, unlike in the conforming case \cite{Brezzi}, $b_h(\cdot,\cdot, \cdot)$ is {\it not} symmetric with respect to the first and second variables or the first and third variables. The equivalent vector form corresponding to \eqref{wformda}-\eqref{wformdb} is given by:  seek $\Psi_h=(u_h,v_h)\in \mathcal{V}_h$ such that
\begin{equation}\label{vformd}
 A_h(\Psi_h,\Phi)+B_h(\Psi_h,\Psi_h,\Phi)=L_h(\Phi) \fl \Phi \in  \mathcal{V}_h
\end{equation}
 where$\fl\, \Xi=(\xi_{1},\xi_{2}),\Theta=(\theta_{1},\theta_{2})$ and $ \Phi=(\varphi_{1},\varphi_{2})\in \mathcal{V}_h,$
  \begin{align}
      &A_h(\Theta,\Phi):=a_h(\theta_1,\varphi_1)+a_h(\theta_2,\varphi_2)\label{defnAh}, \\
      &B_h(\Xi,\Theta,\Phi):=b_h(\xi_{1},\theta_{2},\varphi_{1})+b_h(\xi_{2},\theta_{1},\varphi_{1})-b_h(\xi_{1},\theta_{1},\varphi_{2})\label{defnBh} \; \; \mbox{and}\\
      &L_h(\Phi):=\sit f\varphi_1\dx.\label{defnLh}
     \end{align} 

The nonconforming finite element formulation corresponding to \eqref{vforml} reads as: for given $G\in\ltsq$, find $\Theta_h \in {\cV}_h$ such that
   \begin{equation}\label{vformld}
      \cA_h(\Theta_h,\Phi)=(G,\Phi) \fl \Phi \in \mathcal{V}_h
   \end{equation}
where  $\cA_h(\Theta_h,\Phi):=A_h(\Theta_h,\Phi)+B_h(\Psi,\Theta_h,\Phi)+B_h(\Theta_h,\Psi,\Phi)$ and $ A_h(\cdot,\cdot),\,B_h(\cdot,\cdot,\cdot)$ are defined  in \eqref{defnAh} and \eqref{defnBh}, respectively.

\subsection{Auxiliary Results}

In this subsection, some auxiliary results which are essential for the analysis are stated.

\begin{lem} {\it (Integral average)} \cite{Braess}\label{defnia}   The projection $P_e:L^2(T) \map P_0(e)$  defined by 
  $ \displaystyle P_e \varphi=\frac{1}{h_e}\int_e \varphi\ds$,
satisfies 
\begin{equation}\label{ia}
\|\varphi - P_e \varphi \|_{0,e} \le C h_T^{1/2} |\varphi|_{1,T} \qquad \forall \varphi \in H^1(T).
\end{equation}
\end{lem}

\begin{lem}{(Interpolant)}\cite{Ciarlet, ScbSungZhang, LasLes}\label{interpolant}
Let $\Pi_h:V\map V_h$ be the Morley interpolation operator defined by:
\begin{align*}
&(\Pi_h\varphi)(p)=\varphi(p)\fl p\in\mathfrak{V}_h,\\
&\int_e\frac{\partial \Pi_h \varphi}{\partial \nu}\ds=\int_e\frac{\partial \varphi}{\partial \nu}\ds \fl e\in\mathfrak{E}_h.
\end{align*}
 Then for $\varphi\in H^{2+\alpha}(\Omega)$, $\alpha\in(0,1]$, it holds:
\begin{align*}
&\|\varphi-\Pi_h\varphi\|_{m,p,h} \leq C h^{1+\alpha-m+\frac{2}{p}}\|\varphi\|_{2+\alpha},\qquad 0\leq m\leq 2, 1\leq p< \infty.
\end{align*}
\end{lem}

 For simplicity of notation, the interpolant of  $\Phi \in \cV$ is denoted by $\Pi_h \Phi$ and belongs to $\cV_h$.

\begin{lem}{ (Enrichment function)}\label{enrich1}\cite{ScbSungZhang}
    Let $V_c$ be chosen as {\it Hsieh-Clough-Tocher} macro element space \cite{ScbSungZhang,Ciarlet} which is a conforming relative of the Morley finite element space $V_h$. For any $\varphi \in V_h$, there exists ${E}_h\varphi \in V_c \subset V$ such that 
\begin{align}
\sum_{T \in \mathcal{T}_h} \left( h_T^{-4}|\varphi - E_h\varphi|^2_{0,T} + 
 h_T^{-2}|\varphi - E_h\varphi |^2_{1,T} 
 \right) + |E_h \varphi |^2_{2,h}
&\leq C|\varphi|^2_{2,h}.
 \end{align}
\end{lem}   

Again, for $\Phi \in \cV_h$,  the enrichment function corresponding to $\Phi$   denoted by  $E_h \Phi$,  belongs to $\cV$.

\medskip
In the next lemma, we establish an imbedding result. A similar result has been proved in \cite[Lemma 3.1]{XLR} for the case of {\it convex} polygonal domains. However, for the sake of completeness, we provide a detailed proof for the case of polygonal domains. Note that only the edge estimation in \eqref{intbyparts} is different from the proof  in \cite{XLR}.

\begin{lem}(An imbedding result)\label{imdedding}
For $\varphi\in V_h$, it holds:
\begin{equation*}
|\varphi|_{1,4,h}\leq C|\varphi|_{2,h}.
\end{equation*}
\end{lem}
\begin{proof}
The tangential and normal derivative of $\varphi\in V_h$ are continuous at the midpoint of each edges of $T\in\cT_h$. That is $\varphi_x,\,\varphi_y\in S_h$ where $S_h$ is the nonconforming Crouzeix-Raviart finite element space defined by
 \begin{align*}
      S_h:=\Big\{& w \in \lt\, :\, w|_{T}\in P_1(T)\fl T\in \mathcal{T}_h,\, w \text{ is continuous at the midpoints of}\\
      &\text{ the triangle edges and }
       w =0\text{ at the midpoint of the edges on }\partial\Omega\Big\}.
    \end{align*}
It is enough to prove
$
|w|_{0,4,h}\leq |w|_{1,h}\fl w\in S_h.
$

\noindent Consider the auxiliary problem: given $\theta \in H^{-1}(\Omega)$, seek $\xi$ such that 
\begin{align}\label{auxlap}
-\Delta \xi&=\theta \text{  in } \Omega,\quad\xi=0\text{  on } \partial\Omega.
\end{align}
The solution satisfies the following $a~priori$ bounds 
\begin{equation}\label{apb}
\|\xi\|_1\leq  C\|\theta\|_{-1},\; \|\xi\|_{1+\gamma}\leq C\|\theta\|, 
\end{equation}
where $\gamma\in (\half,1]$ denotes the elliptic regularity of the problem \eqref{auxlap}.
Let $I_h\xi\in S_h$ be an interpolant  which satisfies the estimate \cite{B, Brenner}
\begin{equation}\label{CRIh}
|\xi-I_h\xi|_{0,h}+h|\xi-I_h\xi|_{1,h}\leq Ch^{1+\gamma}\|\xi\|_{1+\gamma}.
\end{equation}
A multiplication of  \eqref{auxlap} with $w$  and a use of Green's formula leads to 
\begin{equation}\label{intbyparts}
(\theta,w)=(-\Delta \xi,w)=\sum_{T\in\cT_h}(\nabla \xi,\nabla w)-\sum_{T\in \cT_h} \int_{\partial T}\frac{\partial\xi}{\partial\nu} w\ds
\end{equation}
The boundary term can be estimated as follows:
\begin{align*}
\sum_{T\in \cT_h} \int_{\partial T}\frac{\partial\xi}{\partial\nu} w\ds
&=\sum_{T\in \cT_h}\sum_{e\subset \partial T} \int_{e}\frac{\partial\xi}{\partial\nu} (w-P_e w)\ds.
\end{align*}
Since $\displaystyle \int_e(w-P_e w)\ds=0 \; \:\forall e\in\mathfrak{E}_h$ and $\displaystyle \frac{\partial }{\partial \nu} I_h \xi$ is a constant over each edge, we obtain
\begin{align*}
\sum_{T\in \cT_h}\sum_{e\subset \partial T} \int_{e}\frac{\partial\xi}{\partial\nu} (w-P_e w)\ds&=
\sum_{T\in \cT_h}\sum_{e\subset \partial T} \int_{e}\frac{\partial}{\partial\nu} (\xi-I_h\xi)(w-P_e w)\ds\\
&\leq\sum_{T\in \cT_h}\sum_{e\subset \partial T}\|\xi-I_h\xi\|_{1,e}\|w-P_e w\|_{0,e}.
\end{align*}
A use of trace theorem, Lemma~\ref{defnia} and \eqref{CRIh} leads to the estimate
\begin{equation}
\Big{|}-\sum_{T\in \cT_h} \int_{\partial T}\frac{\partial\xi}{\partial\nu} w\ds\Big{|}\leq C h^\gamma\|\xi\|_{1+\gamma}|w|_{1,h}.
\end{equation}
Therefore, the {\it a priori} bounds in \eqref{apb} yields
\begin{align}\label{bdw}
(\theta,w)\leq (|\xi|_1+Ch^\gamma\|\xi\|_{1+\gamma})|w|_{1,h}\leq C(\|\theta\|_{-1}+h^\gamma\|\theta\|)|w|_{1,h}.
\end{align}
A choice of $\theta=w^3$ in \eqref{bdw} leads to
\begin{equation}\label{w04h}
|w|_{0,4,h}^4\leq C(\|w^3\|_{-1}+h^\gamma \|w^3\|)|w|_{1,h}.
\end{equation}
A use of inverse inequality yields
\begin{equation}\label{w3l}
\|w^3\|=\|w\|_{L^6(\Omega)}^3\leq Ch^{-\frac{1}{4}}\|w\|_{L^4(\Omega)}^3.
\end{equation}
Also, \Holder inequality and the  imbedding result $L^4(\Omega)\hookrightarrow H_0^1(\Omega)$ lead to
\begin{equation}\label{w3dual}
(w^3,\xi)=\|w\|_{L^4(\Omega)}^3\|\xi\|_{L^4(\Omega)}\leq C\|w\|_{L^4(\Omega)}^3|\xi|_{1} \Longrightarrow \|w^3\|_{-1}\leq C\|w\|_{L^4(\Omega)}^3.
\end{equation}
Hence, a use of \eqref{w3l} and \eqref{w3dual} in \eqref{w04h} leads to the required result
$$
|w|_{0,4,h}\leq C(1+h^{\gamma-\frac{1}{4}})|w|_{1,h}\leq C|w|_{1,h}.
$$
\end{proof}

The next lemma follows from \cite[Lemmas 4.2 \& 4.3]{ScbSungZhang}.
\begin{lem}(Bounds for $A_h(\cdot,\cdot)$)\label{enrichreg}
(i) Let $\boldsymbol{\chi}\in (H^{2+\alpha}(\Omega))^2$ and $\Phi\in\cV_h$. Then, it holds
\begin{equation*}
A_h(\boldsymbol{\chi},E_h\Phi-\Phi)\leq C h^\alpha\trinl\boldsymbol{\chi}\trinr_{2+\alpha}\trinl\Phi\trinr_{2,h}.
\end{equation*}
(ii) Further, for $\boldsymbol{\chi}\in (H^{2+\alpha}(\Omega))^2$ and $\Phi\in (H^{2}_0(\Omega))^2\cap (H^{2+\alpha}(\Omega))^2 $, it holds
\begin{equation*}
A_h(\boldsymbol{\chi},\Pi_h\Phi-\Phi)\leq C h^{2\alpha}\trinl\boldsymbol{\chi}\trinr_{2+\alpha}\trinl\Phi\trinr_{2+\alpha}.
\end{equation*}
\end{lem}  


A use of the definition of $B_h(\cdot, \cdot,\cdot)$,  generalized \Holder inequality and Lemma~\ref{imdedding} leads to a bound  given by 
\begin{equation} \label{boundBh1}
 B_h(\Xi,\Theta,\Phi)  \leq C_b\trinl \Xi\trinr_{2,h} \trinl \Theta\trinr_{2,h} \trinl \Phi\trinr_{2,h},
\end{equation}
where $C_b$ is a positive constant independent of $h$.

\begin{lem}\label{boundBh2} (A bound for $B_h(\cdot,\cdot,\cdot)$) For $\Xi\in (H^{2+\alpha}(\Omega))^2$ and $\Theta,\Phi\in \cV+\cV_h$, there holds
\begin{equation*}
  B_h(\Xi,\Theta,\Phi)\leq C\trinl \Xi\trinr_{2+\alpha} \trinl \Theta\trinr_ {1,4,h}\trinl \Phi\trinr_{1,h}\leq C\trinl \Xi\trinr_{2+\alpha} \trinl \Theta\trinr_ {2,h}\trinl \Phi\trinr_{1,h}.
\end{equation*}
\end{lem}
\begin{proof} Consider 
\begin{equation}\label{defnbh2}
  b_h(\eta,\chi,\varphi)=\half\sit\left((\eta_{yy}\chi_{x}-\eta_{xy}\chi_{y})\varphi_x+(\eta_{xx}\chi_{y}-\eta_{xy}\chi_{x}
                         )\varphi_y\right)\dx.
 \end{equation}
For $\eta \in H^{2+\alpha}(\Omega)$, a use of generalized \Holder inequality and the imbedding result $H^{2+\alpha}(\Omega)\hookrightarrow W^{2,4}(\Omega)$ leads to an estimate of the first term on the right hand side of \eqref{defnbh2} as
\begin{align*}
\half\Big{|}\sit\eta_{yy}\chi_x\varphi_x \dx\Big{|}
&\leq  \left(\sumt|\eta|_{2,4,T}^4\right)^{\frac{1}{4}}\left(\sumt |\chi|_{1,4,T}^4\right)^{\frac{1}{4}}
  \left(\sumt |\varphi|_{1,2,T}^2\right)^{\frac{1}{2}}\\
&\leq \|\eta\|_{W^{2,4}(\Omega)}\,|\chi|_{1,4,h}\,|\varphi|_{1,h}
 \leq C\|\eta\|_{2+\alpha}\,|\chi|_{1,4,h}\,|\varphi|_{1,h}.
\end{align*}
Similar bounds hold true for the remaining three terms in \eqref{defnbh2}. Hence the required result follows using the definition of $B_h(\cdot,\cdot,\cdot)$ and Lemma~\ref{imdedding}.
\end{proof}

\begin{rem}\label{boundBh3}
Using a proof similar to that of  Lemma~\ref{boundBh2}, it can be deduced that for $\Xi\in (H^{2+\alpha}(\Omega))^2$ and $\Theta,\Phi\in \cV+\cV_h$, there holds
\begin{equation*}
  B_h(\Xi,\Theta,\Phi)\leq C\trinl \Xi\trinr_{2+\alpha} \trinl \Theta\trinr_ {1,h}\trinl \Phi\trinr_{1,4,h}\leq C\trinl \Xi\trinr_{2+\alpha} \trinl \Theta\trinr_ {1,h}\trinl \Phi\trinr_{2,h}.
\end{equation*}
\end{rem}

Using the definition of $b_h(\cdot,\cdot, \cdot)$, an integration by parts and a use of \eqref{ia}, the following lemma holds true.
\begin{lem}(An intermediate result)\label{bhflip12}
For $\eta\in (V\cap H^{2+\alpha}(\Omega))+V_h$ and $\chi,\varphi\in V_h$, it holds
 $$b_h(\eta,\chi,\varphi)=b_h(\chi,\eta,\varphi)+\half\sidt\left(\eta_x\chi_y-\eta_y\chi_x \right)\nabla\varphi\cdot\tau\ds$$
where $\tau$ is the unit tangent to the boundary $\partial T$ of 
the triangle $T$. Moreover, 
\begin{equation}
\forall \eta,\chi,\varphi\in V_h,\quad \half\sidt\left(\eta_x\chi_y-\eta_y\chi_x \right)\nabla\varphi\cdot\tau\ds\leq C h|\eta|_{2,h}\,|\chi|_{1,\infty,h}\,|\varphi|_{2,h}.
\end{equation}
\end{lem}

\begin{rem}\label{flipbound}
A use of Lemma~\ref{bhflip12}, Remark~\ref{boundBh3} and imbedding result $H^{2+\alpha}(\Omega)\hookrightarrow W^{1,\infty}(\Omega)$ leads to: \\

for $\Xi_h,\Phi_h\in\cV_h$ and $\boldsymbol{\xi}\in(H^{2+\alpha}(\Omega))^2$,
\begin{equation}
|B_h(\Xi_h,\boldsymbol{\xi},\Phi_h)|\leq |B_h(\boldsymbol{\xi},\Xi_h,\Phi_h)|+Ch\trinl\Xi_h\trinr_{2,h}\trinl\boldsymbol{\xi}\trinr_{2+\alpha}\trinl\Phi_h\trinr_{2,h}.
\end{equation}
\end{rem}

\medskip
The next lemma which will be used to establish the well posedness of the linearized problem \eqref{vformld}, follows easily under the assumption that $\Psi$ is an isolated solution of \eqref{vform}.
\begin{lem}(Well posedness of dual problem)\label{ctsdual}
If $\Psi$ is an isolated solution of \eqref{vform}, then the dual problem defined by: given $Q\in (H^{-1}(\Omega))^2$, find $\boldsymbol{\zeta}\in \cV$ such that
\begin{equation}\label{ctsaux}
\cA(\Phi,\boldsymbol\zeta)=(Q,\Phi) \fl \Phi\in\cV
\end{equation}
is well posed and satisfies the  $a~priori$ bounds:
\begin{equation}\label{apriori23}
\trinl\boldsymbol{\zeta}\trinr_2\leq C\trinl Q\trinr_{-1}, \quad \trinl\boldsymbol{\zeta}\trinr_{2+\alpha}\leq C\trinl Q\trinr_{-1},
\end{equation}
where $\alpha$ denotes the elliptic regularity index and 
$\displaystyle\trinl Q\trinr_{-1}:=\sup_{\boldsymbol{\varphi}\in (H^1_0(\Omega))^2}\frac{(Q,\boldsymbol{\varphi})}{\trinl\boldsymbol{\varphi}\trinr_1}$.
\end{lem}

Since the Morley finite element space $\cV_h$ is not a subspace of $\cV$ and the discrete form $b_h(\cdot, \cdot, \cdot)$ is non-symmetric with respect to first and second or first and third variables, we encounter additional difficulties in establishing the well posedness of the discrete problem \eqref{vformld} in comparison to the conforming case.

\begin{thm}(Well posedness of discrete linearized problem)\label{wellposeld}
If $\Psi$ is an isolated solution of \eqref{vform}, then for sufficiently small $h$, the discrete linearized problem \eqref{vformld} is well-posed.
\end{thm}
\begin{proof}

The space $V_h$ being finite dimensional, uniqueness of solution of \eqref{vformld} implies existence of solution. Uniqueness follows if an $a~priori$ bound for the solution of \eqref{vformld} can be established. That is, we aim to prove that
\begin{equation}
\trinl\Theta_h\trinr_{2,h}\leq C\trinl G\trinr
\end{equation}
for sufficiently small $h$. For $\Phi\in\cV_h$, using Lemma~\ref{boundBh2} and Remark \ref{flipbound}, the following G\aa{}rding's type inequality holds true:
\begin{align}
\cA_h(\Phi,\Phi)&= A_h(\Phi,\Phi)+B_h(\Psi,\Phi,\Phi)+B_h(\Phi,\Psi,\Phi)\notag\\
&\geq\trinl\Phi\trinr_{2,h}^2-C\trinl\Psi\trinr_{2+\alpha}\trinl\Phi\trinr_{2,h}\trinl\Phi\trinr_{1,h}-Ch\trinl\Phi\trinr_{2,h}\trinl\Psi\trinr_{2+\alpha}\trinl\Phi\trinr_{2,h}.\label{garding}
\end{align}
Substitute $\Phi=\Theta_h$ in \eqref{vformld} and use \eqref{garding} to obtain
\begin{equation}\label{Thetabdd}
\trinl\Theta_h\trinr_{2,h}\leq C(h\trinl\Psi\trinr_{2+\alpha}\trinl\Theta_h\trinr_{2,h}+\trinl\Psi\trinr_{2+\alpha}\trinl\Theta_h\trinr_{1,h}+\trinl G\trinr).
\end{equation}
Note that 
\begin{equation}\label{Ehuse}
\trinl \Theta_h\trinr_{1,h}\leq \trinl \Theta_h-E_h\Theta_h\trinr_{1,h}+\trinl E_h\Theta_h\trinr_{1,h}\leq Ch\trinl\Theta_h\trinr_{2,h}+\trinl E_h\Theta_h\trinr_{1}.
\end{equation}
Now we estimate $\trinl E_h\Theta_h\trinr_{1}$. Choose $Q=-\Delta E_h\Theta_h$ and $\Phi=E_h\Theta_h$ in \eqref{ctsaux} and use \eqref{vformld} to obtain
\begin{align}
\trinl E_h\Theta_h\trinr_{1}^2&=\cA(E_h\Theta_h,\boldsymbol{\zeta})=\cA_h(E_h\Theta_h,\boldsymbol{\zeta}-\Pi_h\boldsymbol{\zeta})+\cA_h(E_h\Theta_h,\Pi_h\boldsymbol{\zeta})\notag\\
&=\cA_h(E_h\Theta_h,\boldsymbol{\zeta}-\Pi_h\boldsymbol{\zeta})+\cA_h(E_h\Theta_h-\Theta_h,\Pi_h\boldsymbol{\zeta})+(G,\Pi_h\boldsymbol{\zeta})\notag\\
&=A_h(E_h\Theta_h-\Theta_h,\boldsymbol{\zeta})+A_h(\Theta_h,\boldsymbol{\zeta}-\Pi_h\boldsymbol{\zeta})+B_h(\Psi,E_h\Theta_h,\boldsymbol{\zeta}-\Pi_h\boldsymbol{\zeta})+B_h(E_h\Theta_h,\Psi,\boldsymbol{\zeta}-\Pi_h\boldsymbol{\zeta})\notag\\
&\quad+B_h(\Psi,E_h\Theta_h-\Theta_h,\Pi_h\boldsymbol{\zeta})+B_h(E_h\Theta_h-\Theta_h,\Psi,\Pi_h\boldsymbol{\zeta})+(G,\Pi_h\boldsymbol{\zeta})\label{Ehwellbd}
\end{align}

A use of  Lemmas~\ref{interpolant}, \ref{enrich1}, \ref{enrichreg}, \eqref{boundBh1}, \eqref{apriori23}, Remarks~\ref{boundBh3} and \ref{flipbound} leads to 
\begin{align*}
\trinl E_h\Theta_h\trinr_{1}^2&\leq C\left(h^{\alpha}\trinl\Theta_h\trinr_{2,h}\trinl\boldsymbol{\zeta}\trinr_{2+\alpha}+h^{\alpha}\trinl\Psi\trinr_2\trinl\Theta_h\trinr_{2,h}\trinl\boldsymbol{\zeta}\trinr_{2+\alpha}+h\trinl\Psi\trinr_{2+\alpha}\trinl\Theta_h\trinr_{2,h}\trinl\boldsymbol{\zeta}\trinr_2+\trinl G\trinr\trinl\boldsymbol{\zeta}\trinr_2\right)\\
&\leq C\left(h^\alpha\trinl\Theta_h\trinr_{2,h}+\trinl G\trinr\right)\trinl -\Delta E_h\Theta_h\trinr_{-1}\leq C\left(h^{\alpha}\trinl\Theta_h\trinr_{2,h}+\trinl G\trinr\right)\trinl E_h\Theta_h\trinr_{1}.
\end{align*}
Therefore,
\begin{equation}\label{Ehestimate}
\trinl E_h\Theta_h\trinr_{1,h}\leq C(h^{\alpha}\trinl\Theta_h\trinr_{2,h}+\trinl G\trinr).
\end{equation}
Now, \eqref{Thetabdd}-\eqref{Ehestimate} yield
\begin{align*}
&\trinl\Theta_h\trinr_{2,h}\leq  C_{*}h^{\alpha}\trinl\Theta_h\trinr_{2,h}+C\trinl G\trinr.
\end{align*}
That is,  $\trinl\Theta_h\trinr_{2,h}\leq C\trinl G\trinr$
for a choice of $h\leq h_1=(\frac{1}{2C_*})^{\frac{1}{\alpha}}$ with $\alpha\in (\half,1]$.
\end{proof}
\begin{rem}\label{wellposelddual} 
If $\Psi$ is an isolated solution of \eqref{vform}, then for sufficiently small $h$, the discrete linearized dual problem: given $\mathcal{G}\in \ltsq$, find $\boldsymbol{\zeta_h}\in\cV_h$ such that
\begin{equation}\label{vformlddual}
\cA_h(\Phi,\boldsymbol\zeta_h)=(\mathcal{G},\Phi) \fl \Phi\in\cV_h
\end{equation}
is well posed. The proof is similar to that of Theorem ~\ref{wellposeld} and hence is skipped.
\end{rem}

\section{Existence, Uniqueness and Error Estimates}
\label{sec:ee}

In view of Theorem \ref{wellposeld} and Remark~\ref{wellposelddual}, the bilinear form $\cA_h(\cdot,\cdot): \cV_h \times \cV_h \rightarrow {\mathbb R}$ defined by 
\begin{equation} \label{nonsing}
    \cA_h(\Theta,\Phi)= A_h(\Theta,\Phi)+B_h(\Psi,\Theta,\Phi)+B_h(\Theta,\Psi,\Phi)
\end{equation} 
is nonsingular on $\cV_h\times\cV_h$. 

\medskip
The next lemma establishes that the perturbed bilinear form $\tilde{\cA}_h(\cdot,\cdot)$, constructed using $\Pi_h\Psi$ is also nonsingular. Though a similar result is proved in \cite{Brezzi} for the conforming case, we provide a proof here for the sake of completeness.
  
\begin{lem}{(Nonsingularity of perturbed bilinear form)}\label{nonsingular} Let
$\Pi_h\Psi$ be the interpolation of $\Psi$ as defined in Lemma~\ref{interpolant}.
Then, for sufficiently small $h$, the perturbed bilinear form defined by  
\begin{equation}\label{defnAht}
\tilde \cA_h(\Theta,\Phi)= A_h(\Theta,\Phi)+B_h(\Pi_h\Psi,\Theta,\Phi)+B_h(\Theta,\Pi_h\Psi,\Phi)
\end{equation}
is nonsingular on $\cV_h\times\cV_h$, if \eqref{nonsing} is nonsingular on $\cV_h\times\cV_h$. 
\end{lem}
\begin{proof} The bilinear form 
$\cA_h:\cV_h\times\cV_h\map \bR$ is bounded and satisfies
\begin{align*}
&\sup_{\trinl\Theta \trinr_{2,h}=1} \cA_h(\Theta,\Phi)\geq \beta\trinl\Phi\trinr_{2,h}, \quad \sup_{\trinl\Phi\trinr_{2,h}=1}\cA_h(\Theta,\Phi)\geq \beta\trinl\Theta\trinr_{2,h},
\end{align*}
where $\beta >0$ is a constant.
For $\tilde{\Psi} \in \cV+\cV_h$, a use of the above properties of  $\cA_h(\cdot,\cdot)$ and continuity of $B_h(\cdot,\cdot,\cdot)$ (see \eqref{boundBh1}) yields
\begin{align*}
&\sup_{\trinl\Phi \trinr_{2,h}=1} A_h(\Theta,\Phi)+B_h(\Psi-\tilde{\Psi},\Theta,\Phi)+B_h(\Theta,\Psi-\tilde{\Psi},\Phi)\\
&\geq \sup_{\trinl\Phi\trinr_{2,h}=1} \cA_h(\Theta,\Phi)-\sup_{\trinl\Phi\trinr_{2,h}=1} \left(B_h(\tilde{\Psi},\Theta,\Phi)+B_h(\Theta,\tilde{\Psi},\Phi)\right)\\
&\geq \beta\trinl\Theta_h\trinr_{2,h}-2C_b\trinl  {\tilde \Psi}\trinr_{2,h}\trinl\Theta\trinr_{2,h}
\geq \frac{\beta}{2}\trinl\Theta\trinr_{2,h},
\end{align*}
provided $\displaystyle\trinl {\tilde \Psi} \trinr_{2,h}\leq\frac{\beta}{4C_b}$. Such a choice is justified for sufficiently small $h\leq h_2$ (say), by setting $\tilde{\Psi}=\Psi-\Pi_h\Psi$ and using Lemma~\ref{interpolant}. 
Similarly,  $ \displaystyle \sup_{\trinl\Theta \trinr_{2,h}=1} \tilde{\cA}_h(\Theta ,\Phi)\geq \frac{\beta}{2}\trinl\Phi\trinr_{2,h}\fl \Phi\in\cV_h.$  Hence the required result.
\end{proof}

\subsection{Existence and Local Uniqueness Results}
Consider the nonlinear operator $\mu:\cV_h\map\cV_h$ defined by 
\begin{equation}\label{defnmu}
\tilde{\cA}_h(\mu(\Theta),\Phi)=L_h(\Phi)+B_h(\Pi_h\Psi,\Theta,\Phi)+B_h(\Theta,\Pi_h\Psi,\Phi)-B_h(\Theta,\Theta,\Phi)\fl \Phi\in\cV_h.
\end{equation} 
A use of Lemma~\ref{nonsingular} leads to the fact that the mapping $\mu$ is well-defined and continuous. Also, any fixed point of $\mu$ is a solution of \eqref{vformd} and vice-versa. Hence, in order to show the existence of a solution to \eqref{vformd}, we will prove that the mapping $\mu$ has a fixed point. 
 As a first step to this, define $\bB_R(\Pi_h\Psi):=\left\{\Phi \in\cV_h: \trinl\Phi-\Pi_h\Psi\trinr_{2,h}\leq R \right\}$.
  
\begin{thm}(Mapping of ball to ball) \label{mapball2ball}
For a sufficiently small choice of $h$, there exists a positive constant $R(h)$ such that for any $\Theta \in \cV_h$,
\begin{equation*}
\trinl \Theta-\Pi_h\Psi\trinr_{2,h}\leq R(h)\Rightarrow \trinl\mu(\Theta)-\Pi_h\Psi\trinr_{2,h}\leq R(h).
\end{equation*} 
That is, $\mu$ maps the ball $\bB_{R(h)}(\Pi_h\Psi)$ to itself.
\end{thm}
\begin{proof}
Since  the bilinear form $ \tilde{\cA}_h(\cdot, \cdot)$ is nonsingular, from Lemma \ref{nonsingular}, there exists $\bphi \in\cV_h$ such that $\trinl\bphi \trinr_{2,h}=1$ and 
\begin{align*}
&\frac{\beta}{4}\trinl\mu(\Theta)-\Pi_h\Psi\trinr_{2,h}\leq \tilde\cA_h(\mu(\Theta)-\Pi_h\Psi, \bphi)
\end{align*}
Let $E_h\bphi$ be an enrichment of $\bphi$ (see Lemma \ref{enrich1}).   A use of \eqref{defnAht}, \eqref{defnmu} and \eqref{vform} yields
\begin{align}
&\tilde\cA_h(\mu(\Theta)-\Pi_h\Psi, \bphi)
=\tilde\cA_h(\mu(\Theta), \bphi)-\tilde\cA_h(\Pi_h\Psi, \bphi)\notag\\
&\quad=L_h(\bphi)+B_h(\Pi_h\Psi,\Theta,\bphi)+B_h(\Theta,\Pi_h\Psi,\bphi)-B_h(\Theta,\Theta,\bphi)-A_h(\Pi_h\Psi,\bphi)-2B_h(\Pi_h\Psi,\Pi_h\Psi,\bphi)\notag\\
&\quad=L_h(\bphi-E_h\bphi) +\left(A_h(\Psi, E_h\bphi)-A_h(\Pi_h\Psi,\bphi)\right)+\left(B_h(\Psi,\Psi, E_h\bphi)-B_h(\Pi_h\Psi,\Pi_h\Psi,\bphi)\right)\notag\\
&\qquad+B_h(\Pi_h\Psi-\Theta,\Theta-\Pi_h\Psi,\bphi)=:T_1+T_2+T_3+T_4. \label{Aht_est}
\end{align}
Now we estimate $\{T_i\}_{i=1}^4$. $T_1$ can be estimated using Lemma~\ref{enrich1} and the continuity of $L_h$.
Using Lemma~\ref{enrichreg}, continuity of $A_h(\cdot,\cdot)$ and Lemma~\ref{interpolant}, we obtain
\begin{align*}
T_2 &\leq |A_h(\Psi, E_h\bphi)-A_h(\Pi_h\Psi,\bphi)|\leq |A_h(\Psi,E_h\bphi-\bphi)|+|A_h(\Psi-\Pi_h\Psi,\bphi)|\leq Ch^\alpha \trinl\Psi\trinr_{2+\alpha}.
\end{align*}
A use of Lemmas~\ref{boundBh2}, \ref{enrich1}, \ref{interpolant} and \eqref{boundBh1} leads to
\begin{align*}
T_3 &\leq | B_h(\Psi,\Psi,E_h\bphi)-B_h(\Pi_h\Psi,\Pi_h\Psi,\bphi)|\\
&\leq |B_h(\Psi,\Psi,E_h\bphi-\bphi)-B_h(\Pi_h\Psi-\Psi,\Pi_h\Psi,\Phi)-B_h(\Psi,\Pi_h\Psi-\Psi,\bphi)|\\
&\leq Ch\trinl\Psi\trinr_{2+\alpha}\,\trinl\Psi\trinr_2\,\trinl E_h\bphi\trinr_{2,h}+Ch^{\alpha} \trinl\Psi\trinr_{2+\alpha}\,\trinl\Psi\trinr_2\,\trinl\bphi\trinr_{2,h}\\
&\leq Ch^\alpha\trinl\Psi\trinr_2\trinl\Psi\trinr_{2+\alpha}.
\end{align*}
Finally,  $T_4$ is estimated using \eqref{boundBh1} as
\begin{align*}
T_4&\leq|B_h(\Pi_h\Psi-\Theta,\Theta-\Pi_h\Psi,\bphi)|\leq C\trinl\Theta-\Pi_h\Psi\trinr_{2,h}^2.
\end{align*}
A substitution of the estimates derived for $T_1, T_2, T_3$ and $T_4$ in \eqref{Aht_est} and an appropriate grouping of the terms yields
\begin{align}\label{muestimate}
&\trinl\mu(\Theta)-\Pi_h\Psi\trinr_{2,h}\leq C_1\left(h^{\alpha}+\trinl\Theta-\Pi_h\Psi\trinr_{2,h}^2\right) 
\end{align}
for some positive constants $C_1$ independent of $h$ but dependent on $\trinl\Psi\trinr_{2+\alpha}$. A choice of $h\leq h_3$, where $h_3=\left(\frac{1}{4C_1^2}\right)^\frac{1}{\alpha}$, yields $4C_1^2 h^\alpha\leq 1$. Since $\trinl \Theta-\Pi_h\Psi\trinr_{2,h}\leq R(h)$, for $h\leq h_3$, a choice of $R(h):=2C_1 h^\alpha$ leads to
\begin{align*}
&\trinl\mu(\Theta)-\Pi_h\Psi\trinr_{2,h}\leq C_1 h^\alpha\left(1+4C_1^2 h^\alpha\right)\leq R(h)
\end{align*}
This completes the proof.
\end{proof}
\begin{thm}(Existence)\label{exitenceuniqueness}
For sufficiently small $h$, there exists a solution $\Psi_h$ of the discrete problem \eqref{vformd} that satisfies $\trinl \Psi_h-\Pi_h\Psi\trinr_{2,h}\leq R(h)$, for some positive constant $R(h)$ depending on $h$.  
\end{thm}
\begin{proof}
Lemma~\ref{mapball2ball} leads to the fact that $\mu$ maps the ball $\bB_{R(h)}(\Pi_h\Psi)$ to itself. Therefore, an application of Schauder fixed point theorem~\cite{Kesavan} yields that the mapping $\mu$ has a fixed point, say $\Psi_h$. Hence, $\Psi_h$ is an approximate solution of \eqref{vformd} which satisfies $\trinl\Psi_h-\Pi_h\Psi\trinr_{2,h}\leq R(h)$.
\end{proof}

\begin{thm}(Contraction result)\label{contractionthm}
For $\Theta_1,\Theta_2\in \bB_{R(h)}(\Pi_h\Psi)$ with $R(h)$ as defined in Theorem \ref{mapball2ball}, the following contraction result holds true:
 \begin{equation}\label{contractioneqn}
 \trinl\mu(\Theta_1)-\mu(\Theta_2)\trinr_{2,h}\leq Ch^{\alpha}\trinl\Theta_1-\Theta_2\trinr_{2,h},
 \end{equation}
for some positive constant $C$ independent of $h$.
\end{thm}

\begin{proof}
For $\Theta_1,\Theta_2\in  \bB_{R(h)}(\Pi_h\Psi)$, let $\mu(\Theta_i), i=1,2$ be the solutions of:
\begin{align}\label{mueqn1}
\tilde\cA_h(\mu(\Theta_i),\Phi)&=L_h(\Phi)+B_h(\Pi_h\Psi,\Theta_i,\Phi)+B_h(\Theta_i,\Pi_h\Psi,\Phi)-B_h(\Theta_i,\Theta_i,\Phi)\fl\Phi\in\cV_h.
\end{align}
The nonsingularity of $\tilde\cA_h(\cdot,\cdot)$ yields a $\bar{\Phi}$ with $\trinl\bar\Phi\trinr_{2,h}=1$.  With \eqref{mueqn1} and \eqref{boundBh1}, we obtain
\begin{align*}
&\frac{\beta}{4}\trinl\mu(\Theta_1)-\mu(\Theta_2)\trinr_{2,h}\leq \tilde\cA_h(\mu(\Theta_1)-\mu(\Theta_2),\bar{\Phi})\notag\\
&=B_h(\Pi_h\Psi,\Theta_1-\Theta_2,\bar\Phi)+B_h(\Theta_1-\Theta_2,\Pi_h\Psi,\bar\Phi)+B_h(\Theta_2,\Theta_2,\bar\Phi)-B_h(\Theta_1,\Theta_1,\bar\Phi)\\
&=B_h(\Theta_2-\Theta_1,\Theta_1-\Pi_h\Psi,\bar{\Phi})+B_h(\Theta_2-\Pi_h\Psi,\Theta_2-\Theta_1,\bar{\Phi})\\
&\leq C\trinl\Theta_2-\Theta_1\trinr_{2,h}\left(\trinl\Theta_1-\Pi_h\Psi\trinr_{2,h}+
\trinl\Theta_2-\Pi_h\Psi\trinr_{2,h}\right).
\end{align*}
Since $\Theta_1,\Theta_2\in  \bB_{R(h)}(\Pi_h\Psi)$, for a choice of $R(h)$ as in the proof of  Theorem \ref{mapball2ball}, for sufficiently small $h$, we obtain
\begin{equation}
\trinl\mu(\Theta_1)-\mu(\Theta_2)\trinr_{2,h}\leq Ch^{\alpha} \trinl\Theta_2-\Theta_1\trinr_{2,h},
\end{equation}
for some positive constant $C$ independent of $h$. This completes the proof.
\end{proof}

\begin{rem}(Local uniqueness) Let $\Psi$  be an isolated solution of \eqref{vform}. For sufficiently small choice of $h$, Theorem~\ref{contractionthm} establishes the local uniqueness of the solution of \eqref{vformd}.
\end{rem}
\medskip
\subsection{Error Estimates}
In this subsection, the error estimates in the broken energy and $H^1$ norms are established.
\begin{thm}{(Energy norm estimate)}\label{eetimate}
Let $\Psi$ and $\Psi_h$ be the solutions of \eqref{vform} and \eqref{vformd} respectively. Under the assumption that $\Psi$ is an isolated solution, for sufficiently small $h$, it holds
\begin{equation}
\trinl\Psi-\Psi_h\trinr_{2,h}\leq C h^{\alpha},
\end{equation}
where $\alpha\in (\half,1]$ is the index of elliptic regularity.
\end{thm}
\begin{proof}
A use of triangle inequality yields
\begin{equation}\label{newnn}
\trinl\Psi-\Psi_h\trinr_{2,h}\leq \trinl\Psi-\Pi_h\Psi\trinr_{2,h}+\trinl\Pi_h\Psi-\Psi_h\trinr_{2,h}.
\end{equation}
For sufficiently small $h$, Theorem~\ref{exitenceuniqueness} leads to
\begin{equation}\label{pihsoln}
\trinl\Pi_h\Psi-\Psi_h\trinr_{2,h}\leq Ch^\alpha.
\end{equation}
Now, Lemma~\ref{interpolant} , \eqref{pihsoln}  and \eqref{newnn} establish the required estimate.
\end{proof}

\begin{thm}{($H^1$ estimate)}\label{h1estimate}
Let $\Psi$ and $\Psi_h$ be the solutions of \eqref{vform} and \eqref{vformd} respectively. Assume that $\Psi$ is an isolated solution. Then, for sufficiently small $h$, it holds
\begin{equation}
\trinl\Psi-\Psi_h\trinr_{1,h}\leq C h^{2\alpha},
\end{equation}
where $\alpha\in (\half,1]$ is the index of elliptic regularity.
\end{thm}
\begin{proof} A use of triangle inequality yields
\begin{equation}\label{trienq}
\trinl\Psi-\Psi_h\trinr_{1,h}\leq \trinl\Psi-\Pi_h\Psi\trinr_{1,h}+\trinl\Pi_h\Psi-\Psi_h\trinr_{1,h}\leq \trinl\Psi-\Pi_h\Psi \trinr_{1,h}+\trinl\boldsymbol{\rho}-E_h\boldsymbol{\rho}\trinr_{1,h}+\trinl E_h\boldsymbol{\rho}\trinr_{1},
\end{equation}
 where $\boldsymbol{\rho}=\Pi_h\Psi-\Psi_h$. A choice of $Q=-\Delta E_h\boldsymbol{\rho}$ and $\Phi=E_h\boldsymbol{\rho}$ in the dual problem~\eqref{ctsaux}  and a use of \eqref{vform}, \eqref{vformd} leads to
\begin{align}
\left(\nabla E_h\boldsymbol{\rho},\nabla E_h\boldsymbol{\rho}\right) &= \cA_h({E_h\boldsymbol{\rho},\boldsymbol \zeta})=\cA_h(E_h\boldsymbol{\rho}-\boldsymbol{\rho},{\boldsymbol \zeta})+\cA_h(\boldsymbol{\rho},{\boldsymbol \zeta})\nonumber\\
&= A_h(E_h\boldsymbol{\rho}-\boldsymbol{\rho},{\boldsymbol \zeta})+B_h(\Psi,E_h\boldsymbol{\rho}-\boldsymbol{\rho},{\boldsymbol \zeta})+B_h(E_h\boldsymbol{\rho}-\boldsymbol{\rho},\Psi,{\boldsymbol \zeta})\nonumber\\
&\quad+A_h(\Pi_h\Psi-\Psi,{\boldsymbol\zeta})+A_h(\Psi-\Psi_h,\boldsymbol{\zeta}-\Pi_h\boldsymbol{\zeta})+A_h(\Psi,\Pi_h\boldsymbol{\zeta}-\boldsymbol{\zeta})+L_h(\boldsymbol{\zeta}-\Pi_h\boldsymbol{\zeta})\nonumber\\
&\quad
+\left(B_h(\Psi,\Pi_h\Psi-\Psi_h,{\boldsymbol \zeta})+B_h(\Pi_h\Psi-\Psi_h,\Psi,{\boldsymbol \zeta})-B_h(\Psi,\Psi,{\boldsymbol \zeta})+B_h(\Psi_h,\Psi_h,\Pi_h{\boldsymbol \zeta})\right)\nonumber\\
&=:\sum_{i=1}^{8} T_i.
\end{align}

$T_1$ is estimated using Lemma \ref{enrichreg} and \eqref{pihsoln}. $T_4$ and $T_6$ are estimated using Lemma~\ref{enrichreg}. $T_5$ is estimated using continuity of $A_h(\cdot,\cdot)$, Lemma~\ref{interpolant} and Theorem~\ref{eetimate}. The term $T_7$ is estimated using continuity of $L_h$ and Lemma~\ref{interpolant}.
$T_2$ is estimated using Remark~\ref{boundBh3}, Lemma~\ref{enrich1} and \eqref{pihsoln} as
\begin{align}\label{T2}
T_2&\leq |B_h(\Psi,E_h\boldsymbol{\rho}-\boldsymbol{\rho},{\boldsymbol \zeta})|
\leq C\trinl\Psi\trinr_{2+\alpha} \trinl E_h\boldsymbol{\rho}-\boldsymbol{\rho}\trinr_{1,h}\trinl{\boldsymbol \zeta}\trinr_2\leq C h^{1+\alpha}\trinl\Psi\trinr_{2+\alpha}\trinl{\boldsymbol \zeta}\trinr_2.
\end{align}
$T_3$ is estimated using Remark~\ref{flipbound}, Lemma~\ref{enrich1}, \eqref{T2} and \eqref{pihsoln} as
\begin{equation*}
T_3\leq|B_h(E_h\boldsymbol{\rho}-\boldsymbol{\rho},\Psi,{\boldsymbol \zeta})|\leq |B_h(\Psi,E_h\boldsymbol{\rho}-\boldsymbol{\rho},{\boldsymbol \zeta})|+Ch\trinl E_h\boldsymbol{\rho}-\boldsymbol{\rho}\trinr_{2,h}\trinl\Psi\trinr_{2+\alpha}\trinl{\boldsymbol \zeta}\trinr_2\leq Ch^{1+\alpha}\trinl\Psi\trinr_{2+\alpha}\trinl{\boldsymbol \zeta}\trinr_2.
\end{equation*}
Finally, a use of Remarks \ref{boundBh3} , \ref{flipbound}, Lemmas~\ref{interpolant}, \ref{boundBh2},  Theorem~\ref{eetimate} and  \eqref{boundBh1} yields an estimate for $T_8$ as
\begin{align*}
T_8& =B_h(\Psi,\Pi_h\Psi-\Psi_h,{\boldsymbol \zeta})+B_h(\Pi_h\Psi-\Psi_h,\Psi,{\boldsymbol \zeta})-B_h(\Psi,\Psi,{\boldsymbol \zeta})+B_h(\Psi_h,\Psi_h,\Pi_h{\boldsymbol \zeta})\\
&=B_h(\Psi,\Pi_h\Psi-\Psi,{\boldsymbol \zeta})+B_h(\Pi_h\Psi-\Psi,\Psi,{\boldsymbol \zeta})\\
&\quad+B_h(\Psi,\Psi-\Psi_h,{\boldsymbol \zeta})+B_h(\Psi-\Psi_h,\Psi,{\boldsymbol \zeta})-B_h(\Psi,\Psi,{\boldsymbol \zeta})+B_h(\Psi_h,\Psi_h,\Pi_h{\boldsymbol \zeta})\\
&=B_h(\Psi,\Pi_h\Psi-\Psi,{\boldsymbol \zeta})+B_h(\Pi_h\Psi-\Psi,\Psi,{\boldsymbol \zeta})\\
&\quad+B_h(\Psi-\Psi_h,\Psi-\Psi_h,{\boldsymbol \zeta})+B_h(\Psi_h-\Psi,\Psi_h,\Pi_h{\boldsymbol \zeta}-{\boldsymbol \zeta})+B_h(\Psi,\Psi_h,\Pi_h{\boldsymbol \zeta}-{\boldsymbol \zeta})\\
&\leq Ch^{2\alpha}(\trinl{\boldsymbol \zeta}\trinr_2 +\trinl{\boldsymbol \zeta}\trinr_{2+\alpha}).
\end{align*}
A combination of the estimates $T_1$ to $T_8$ and $a~priori$ bounds \eqref{apriori23} for the linearized dual problem yields
\begin{align}\label{gradEh}
&(\nabla E_h\boldsymbol{\rho},\nabla E_h\boldsymbol{\rho})\leq Ch^{2\alpha}\trinl -\Delta E_h\boldsymbol{\rho}\trinr_{-1}\leq Ch^{2\alpha}\trinl E_h\boldsymbol{\rho}\trinr_{1} \Longrightarrow \trinl E_h\boldsymbol{\rho}\trinr_{1}\leq C h^{2\alpha}.
\end{align}
A use of Lemmas~\ref{interpolant},~\ref{enrich1},  \eqref{pihsoln} and the last statement of \eqref{gradEh} in \eqref{trienq} completes the proof.
\end{proof}

\subsection{Convergence of the Newton's Method}
In this subsection, we define a working procedure to find an approximation for the discrete solution $\Psi_h$. The discrete solution $\Psi_h$ of \eqref{vformd} is characterized by the fixed point of \eqref{defnmu}. This depends on the unknown $\Pi_h\Psi$  and hence the approximate solution for \eqref{vformd} is computed using Newton's method in implementation. The iterates of the Newton's method are defined by
\begin{equation}\label{NewtonIterate}
A_h(\Psi_h^{n},\Phi)+B_h(\Psi_h^{n-1},\Psi_h^{n},\Phi)+B_h(\Psi_h^{n},\Psi_h^{n-1},\Phi)=B_h(\Psi_h^{n-1},\Psi_h^{n-1},\Phi)+L_h(\Phi)\fl \Phi\in \cV_h.
\end{equation} 
Now we establish that these iterates in fact converge quadratically to the solution of \eqref{vformd}. 

\begin{thm}(Convergence of Newton's method)\label{NewtonThm}
Let $\Psi$ be an isolated solution of \eqref{vform} and let  $\Psi_h$ solve \eqref{vformd}. There exists $\rho> 0$, independent of $h$, such that for any initial guess $\Psi_h^0$ which satisfies
 $\displaystyle \trinl \Psi_h^0-\Psi_h\trinr_{2,h}\leq \rho,\:\, \trinl\Psi_h^n-\Psi_h\trinr_{2,h}  \leq\frac{\rho}{2^n}$
holds true.
That is, the iterates of the Newton's method defined in \eqref{NewtonIterate} are well defined and converge quadratically to $\Psi_h$.
\end{thm}
\begin{proof}
From Lemma~\ref{nonsingular}, there exists $\delta>0$ such that for each $Z_h\in\cV_h$ satisfying $\trinl Z_h-\Pi_h\Psi\trinr_{2,h}\leq \delta$, the form
\begin{equation}\label{NewtonNonsingular}
 A_h(\Theta,\Phi)+B_h(Z_h,\Theta,\Phi)+B_h(\Theta,Z_h,\Phi)
\end{equation}
\smallskip
is non singular in $\cV_h\times\cV_h$.
From \eqref{pihsoln}, for sufficiently small $h$, $\trinl\Pi_h\Psi-\Psi_h\trinr_{2,h}\leq Ch^\alpha$. Thus $h$ can be chosen sufficiently small so that $\trinl\Pi_h\Psi-\Psi_h\trinr_{2,h}\leq \frac{\delta}{2}$. Define
\begin{equation}\label{defnrho}
\rho:=\min\left\{\frac{\delta}{2},\frac{\beta}{16C_b}\right\}
\end{equation}
where $\beta$ and $ C_b$ are respectively the coercivity constant of $\cA_h(\cdot,\cdot)$ and boundedness constant of $B_h(\cdot,\cdot,\cdot)$ (see \eqref{boundBh2}).
Assume that the initial guess $\Psi_h^0$ satisfies $\trinl\Psi_h-\Psi_h^0\trinr_{2,h}\leq \rho$. Then
\begin{equation*}
\trinl\Pi_h\Psi-\Psi_h^0\trinr_{2,h}\leq \trinl\Pi_h\Psi-\Psi_h\trinr_{2,h}+\trinl\Psi_h-\Psi_h^0\trinr_{2,h}\leq \delta
\end{equation*}
Since \eqref{NewtonNonsingular} is nonsingular, the first iterate $\Psi_h^1$ of the Newton's method in \eqref{NewtonIterate} is well defined for the initial guess $\Psi_h^0$. Using the nonsingularity of \eqref{NewtonNonsingular}, there exists $\bar\Phi\in\cV_h$ such that $\trinl\bar\Phi\trinr_{2,h}=1$ which satisfies
\begin{equation*}
\frac{\beta}{8}\trinl\Psi_h^1-\Psi_h\trinr_{2,h}\leq A_h(\Psi_h^1-\Psi_h,\bar\Phi)+B_h(\Psi_h^0,\Psi_h^1-\Psi_h,\bar\Phi)+B_h(\Psi_h^1-\Psi_h,\Psi_h^0,\bar\Phi).
\end{equation*}
A use of \eqref{NewtonIterate}, \eqref{vformd}, \eqref{boundBh1} yields 
\begin{align}
&A_h(\Psi_h^1-\Psi_h,\bar\Phi)+B_h(\Psi_h^0,\Psi_h^1-\Psi_h,\bar\Phi)+B_h(\Psi_h^1-\Psi_h,\Psi_h^0,\bar\Phi)\notag\\
&=B_h(\Psi_h^0,\Psi_h^0,\bar \Phi)+L_h(\bar \Phi)-A_h(\Psi_h,\bar \Phi)-B_h(\Psi_h^0,\Psi_h,\bar \Phi)-B_h(\Psi_h,\Psi_h^0,\bar \Phi)\notag\\
&=B_h(\Psi_h^0,\Psi_h^0,\bar \Phi)+B_h(\Psi_h,\Psi_h,\bar \Phi)-B_h(\Psi_h^0,\Psi_h,\bar \Phi)-B_h(\Psi_h,\Psi_h^0,\bar \Phi)\notag\\
&=B_h(\Psi_h^0-\Psi_h,\Psi_h^0-\Psi_h,\bar \Phi)\leq C_b\trinl \Psi_h^0-\Psi_h\trinr_{2,h}^2.\label{induction0}
\end{align}
Hence, $\trinl\Psi_h^1-\Psi_h\trinr_{2,h}\leq \frac{8C_b}{\beta}\trinl \Psi_h^0-\Psi_h\trinr_{2,h}^2$. Since $\trinl\Psi_h^0-\Psi_h\trinr\leq \rho\leq \frac{\beta}{16C_b}$, we obtain
\begin{equation}\label{induction0b}
\trinl\Psi_h^1-\Psi_h\trinr_{2,h}\leq\half\trinl \Psi_h^0-\Psi_h\trinr_{2,h}\leq \frac{\rho}{2}.
\end{equation}
Since $\trinl\Psi_h^1-\Psi_h\trinr_{2,h}\leq \rho$, the form~\eqref{NewtonNonsingular} is nonsingular for $Z_h=\Psi_h^1$. Continuing the process, we obtain
\begin{equation}
\trinl \Psi_h^n-\Psi_h\trinr_{2,h}\leq \frac{\rho}{2^n}.
\end{equation}
Moreover, proceeding as in the proof of the estimate \eqref{induction0}, it can be shown that
\begin{equation}
\trinl\Psi_h^{n+1}-\Psi_h\trinr_{2,h}\leq \left(8C_b/\beta\right)\trinl\Psi_h^n-\Psi_h\trinr_{2,h}^2.
\end{equation}
This establishes that the Newton's method converges quadratically to $\Psi_h$.
This completes the proof.
\end{proof}
\begin{rem}{(Local uniqueness)}
The local uniqueness of solution of \eqref{vformd} also follows from Theorem~\ref{NewtonThm}. We observe that the definition of $\rho$ in \eqref{defnrho} does not depend on $h$. From  Theorem~\ref{NewtonThm}, it is clear that for any initial guess $\Psi_h^0$ which lies in the ball of radius $\rho$ with center at $\Psi_h$, the sequence generated by \eqref{NewtonIterate} will converge uniquely to $\Psi_h$.  In particular, if we choose the initial guess $\Psi_h^0=\Pi_h\Psi$, then the sequence generated by the iterates of the Newton's method will also converge to $\Psi_h$ which shows the local uniqueness of the solution $\Psi_h$.
\end{rem}

\section{Numerical Experiments}
 \label{sec:num}
 
  In this section, two numerical experiments that justify the theoretical results are presented. The  implementations have been carried out in MATLAB. The results illustrate the order of convergence obtained for the numerical solution of \eqref{vke}-\eqref{vkb} computed using the Morley finite element scheme. For a detailed description of construction of basis functions for the Morley element, see Ming \& Xu~\cite{MingXu}. We implement the Newton's method defined in \eqref{NewtonIterate} to solve the discrete problem~\eqref{vformd}.

\subsection{Example 1}\label{example1}
In the first example, we choose the right hand side load functions such that the exact solution is given by
 \begin{align*}
     u(x,y)&=x^2(1-x)^2y^2(1-y)^2;\quad v(x,y)=\sin^2(\pi x)\sin^2(\pi y)
 \end{align*}
on the unit square. The initial triangulation is chosen  as shown in Figure~\ref{fig:MT_h}(a). In the uniform red-refinement process, each triangle $T$ is divided into four similar triangles \cite{AlbertyCarstensenFunken} as in Figure~\ref{fig:MT_h}(b). 
 
Let the mesh parameter at the $N$-th level be denoted by $h_N$ and the computational error by $e_N$. The  experimental order of convergence at the $N$-th level is defined by
       \begin{equation*}
       \alpha_N:=log(e_{N-1}/e_N)/log(h_{N-1}/h_{N})=log(e_{N-1}/e_N)/log(2).
       \end{equation*}

\begin{figure}[h]
\begin{center}
\subfloat[]{\includegraphics[width = 2in]{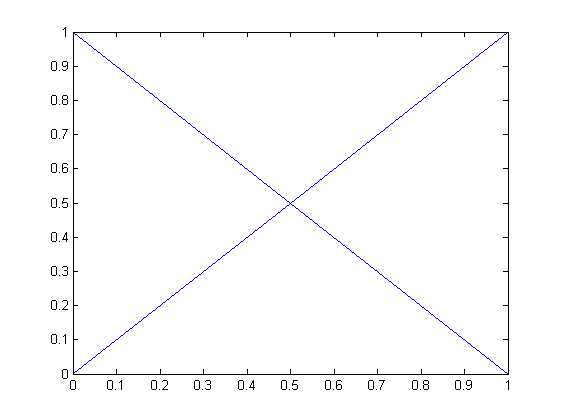}}
     \subfloat[]{\includegraphics[width = 2in]{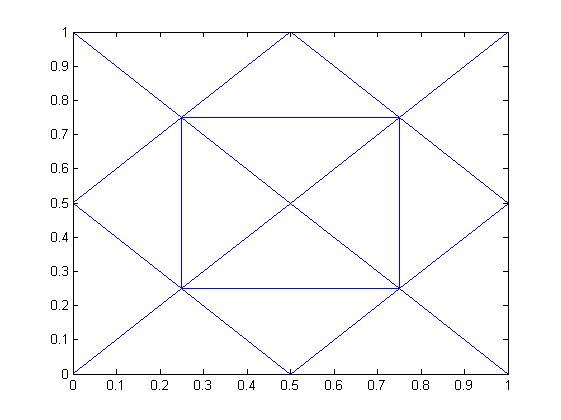}}
     \caption{Initial triangulation $\mathcal{T}_0$ and its red refinement $\mathcal{T}_1$ }
     \label{fig:MT_h}
\end{center}
     
\end{figure}

Tables ~\ref{table:OC_NCFEM_uh} and ~\ref{table:OC_NCFEM_vh} show the errors and experimental convergence rates for the variables $u_h$ and $v_h$. In Figures ~\ref{fig:convrate_uhM}-\ref{fig:convrate_vhM}, the convergence history of the errors in broken energy, $H^1$ and $L^2$ norms are illustrated. The computational order of convergences in broken $H^2,\; H^1$ norms are quasi-optimal and verify the theoretical results obtained in Theorems~\ref{eetimate} and ~\ref{h1estimate} for $\alpha=1$. The order of convergence with respect to $L^2$ norm is sub-optimal justifying the results in~\cite{HuShi} that using a lower order finite element method, the order of convergence in $L^2$ norm cannot be improved than that of the $H^1$ norm.
   
\begin{table}
   \begin{center}
     \begin{tabular}{ | c| c | c |c |c |c |c |}     
      \hline
      \# unknowns & $|u-u_h|_{2,h}$ & Order  & $|u-u_h|_{1,h}$ & Order & $\|u-u_h\|_{L^2}$ & Order\\ 
       \hline \hline
      
     25    &  0.874685E-1  & -         & 0.102155E-1 & -   & 0.386068E-2  & -\\ 
      \hline
     113   &  0.405787E-1  & 1.1080    & 0.257318E-2 & 1.9891   & 0.919743E-3  & 2.0695 \\ 
      \hline
     481   &  0.209921E-1  & 0.9508    & 0.732470E-3 & 1.8127   & 0.248134E-3  & 1.8901\\ 
       \hline
     1985  &  0.106209E-1  & 0.9829    & 0.191118E-3 & 1.9383   & 0.636227E-4  & 1.9635\\ 
       \hline
     8065  &  0.532754E-2  & 0.9953    & 0.483404E-4 & 1.9831   & 0.160158E-4  & 1.9900\\
       \hline 
     32513 &  0.266595E-2  & 0.9988    & 0.121213E-4 & 1.9956   & 0.401107E-5  & 1.9974\\
       \hline
     \end{tabular}
    \end{center}
     \caption{Errors and convergence rates of $u_h$  in broken $H^2,H^1$ and $L^2$ norms}
     \label{table:OC_NCFEM_uh}
 \end{table}
     
 \begin{table}
    \begin{center}
       \begin{tabular}{ |c| c | c |c |c |c |c |}     
         \hline
         \# unknowns & $|v-v_h|_{2,h}$ & Order  & $|v-v_h|_{1,h}$ & Order & $\|v-v_h\|_{L^2}$ & Order\\ 
          \hline \hline
 
          25    & 19.245671 & -   & 2.140613E-0 &-   & 0.770876E-0 &- \\ 
          \hline
          113   & 9.5043699 & 1.0178   & 0.569979E-0 & 1.9090   & 0.177898E-0 & 2.1154 \\ 
          \hline
          481   & 5.0549209 & 0.9109   & 0.161737E-0 & 1.8172   & 0.482777E-1 & 1.8816 \\ 
          \hline
          1985  & 2.5758939 & 0.9726   & 0.421546E-1 & 1.9398   & 0.123930E-1 & 1.9618 \\ 
          \hline
          8065  & 1.2944929 & 0.9926   & 0.106618E-1 & 1.9832   & 0.312076E-2 & 1.9895\\
          \hline  
          32513 & 0.6480848 & 0.9981   & 0.267351E-2 & 1.9956   & 0.781643E-3 & 1.9973\\
          \hline
       \end{tabular}
    \end{center}
    \caption{Errors and convergence rates of $v_h$  in broken $H^2,H^1$ and $L^2$ norms}\label{table:OC_NCFEM_vh}
 \end{table}
   
\begin{figure}[h]
   \begin{center}
    \includegraphics[height=3in,width=5in,angle=0]{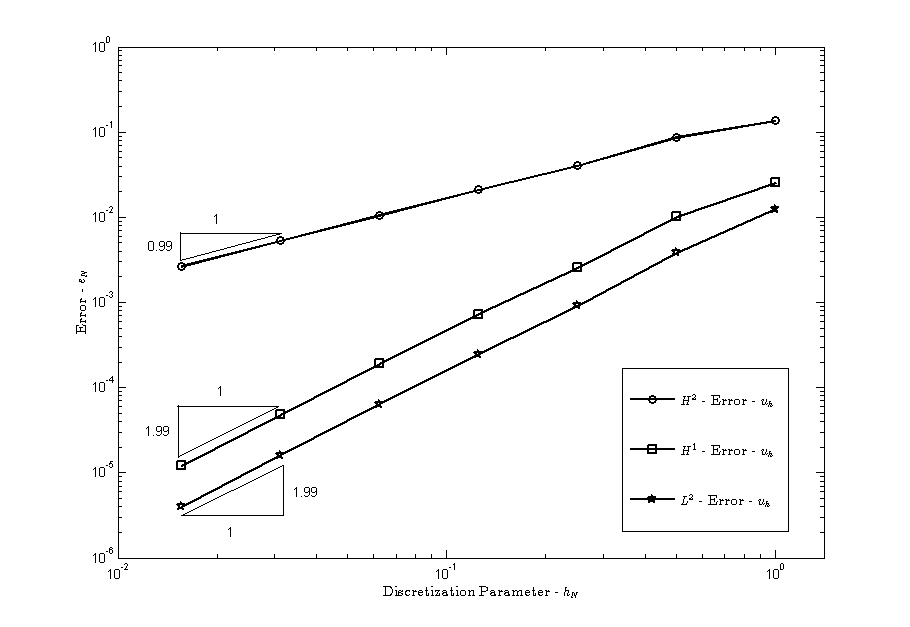}
    \caption{Convergence history of displacement for Example 1}
      \label{fig:convrate_uhM}
   \end{center}
\end{figure}

\begin{figure}[h]
  \begin{center}
   \includegraphics[height=3in,width=5in,angle=0]{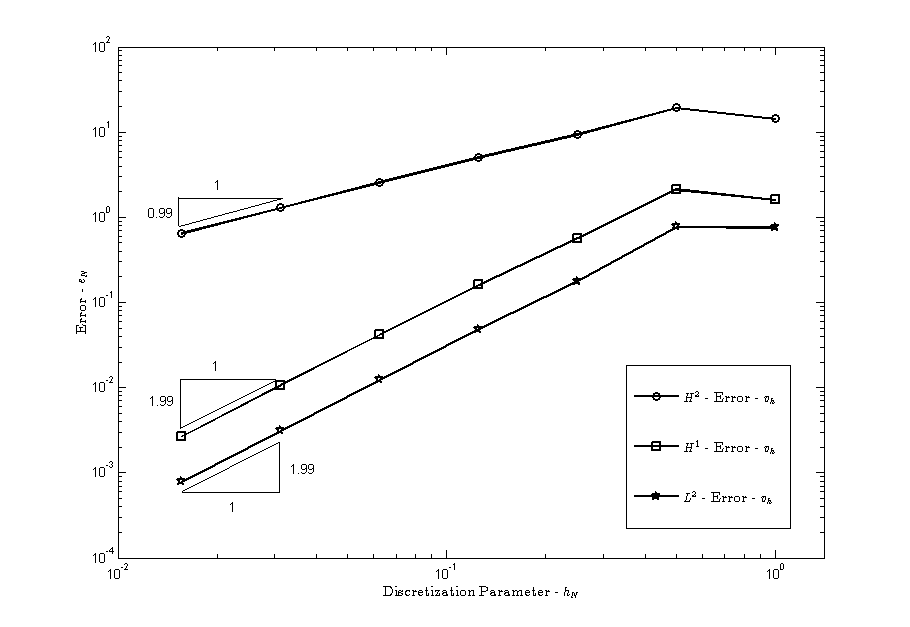}
   \caption{Convergence history of Airy stress for Example 1}
   \label{fig:convrate_vhM}
  \end{center}
\end{figure}

\subsection{Example 2}\label{example2}
Consider the L-shaped domain $\Omega=(-1,1)^2 \setminus([0,1)\times(-1,0])$ (see Figure ~\ref{fig:Lshape}). Choose the right hand functions such that the exact singular solution~\cite{Grisvard} in polar coordinates is given by
\begin{align*}
 u(r,\theta)&=(r^2 cos^2\theta-1)^2 (r^2 sin^2\theta-1)^2 r^{1+\alpha}g_{\alpha,\omega}(\theta);\quad v(r,\theta)=u(r,\theta),
\end{align*}
  where $\omega:=\frac{3\pi}{2}$ and $\alpha:= 0.5444837367$ is a non-characteristic 
  root of $\sin^2(\alpha\omega) = \alpha^2\sin^2(\omega)$ with
  \begin{align*}
    g_{\alpha,\omega}(\theta)=&\left(\frac{1}{\alpha-1}\sin\big{(}(\alpha-1)\omega\big{)}-\frac{1}{\alpha+1}\sin\big{(}(\alpha+1)\omega\big{)}\right)\Big{(}\cos\big{(}(\alpha-1)\theta\big{)}-\cos\big{(}(\alpha+1)\theta\big{)}\Big{)}\\
    &-\left(\frac{1}{\alpha-1}\sin\big{(}(\alpha-1)\theta\big{)}-\frac{1}{\alpha+1}\sin\big{(}(\alpha+1)\theta\big{)}\right)\Big{(}\cos\big{(}(\alpha-1)\omega\big{)}-\cos\big{(}(\alpha+1)\omega\big{)}\Big{)}.
    \end{align*}     
  Tables~\ref{table:OC_Lshape_uh} and~\ref{table:OC_Lshape_vh} show the errors and experimental convergence rates for the variables $u_h$ and $v_h$. The domain being non-convex, we do not obtain linear and quadratic order of convergences in broken energy and $H^1$ norms for displacement and Airy stress functions.
 
 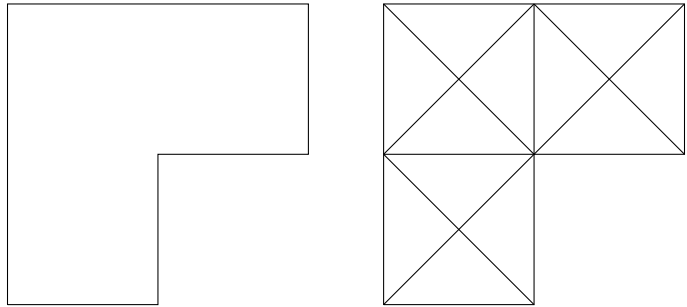
\begin{figure}[h]
    \begin{center}
      \begin{tikzpicture}
       \draw[scale=2]  (0,0)--(1,0)--(1,1)--(2,1)--(2,2)--(0,2)--(0,0);
       \draw[scale=2]  (2.5,0)--(3.5,0)--(3.5,1)--(4.5,1)--(4.5,2)--(2.5,2)--(2.5,0);
       \draw[scale=2]  (2.5,0)--(4.5,2);
       \draw[scale=2]  (2.5,1)--(3.5,2);
       \draw[scale=2]  (2.5,1)--(3.5,1)--(3.5,2);
       \draw[scale=2]  (2.5,1)--(3.5,0);
       \draw[scale=2]  (2.5,2)--(3.5,1);
       \draw[scale=2]  (3.5,2)--(4.5,1);
       
      \end{tikzpicture}
      \caption{L-shaped domain and its initial triangulation}
      \label{fig:Lshape}
    \end{center}
 \end{figure}
 \medskip

 \begin{table}[!h]
    \begin{center}
     \begin{tabular}{ | c| c | c |c |c |c |c |}     
     \hline
     \# unknowns  & $|u-u_h|_{2,h}$ & Order  & $|u-u_h|_{1,h}$ & Order & $\|u-u_h\|_{L^2}$ & Order\\ 
      \hline \hline
    
        17   &  29.209171   & -      & 6.363539E-0  & -        & 2.769499E-0 & -\\ 
      \hline
        81   &  14.130192   & 1.0476 & 1.682747E-0  & 1.9190   & 0.693436E-0 & 1.9977 \\ 
      \hline
       353   &  7.5651300   & 0.9013 & 0.491659E-0  & 1.7750   & 0.200814E-0 & 1.7879 \\ 
      \hline
      1473   &  3.9620126   & 0.9331 & 0.146551E-0  & 1.7462   & 0.583024E-1 & 1.7842\\ 
      \hline
      6017   &  2.0841141   & 0.9267 & 0.487106E-1  & 1.5891   & 0.179703E-1 & 1.6979 \\ 
      \hline
     24321   &  1.1252534   & 0.8891 & 0.187772E-1  & 1.3752   & 0.613474E-2 & 1.5505\\
      \hline
    \end{tabular}
    \end{center}
    \caption{Errors and the experimental convergence rates for $u_h$ in broken $H^2, H^1$ and $L^2$ norms for L-shaped domain}
    \label{table:OC_Lshape_uh}
    \end{table}
    \medskip

    \begin{table}[!h]
       \begin{center}
        \begin{tabular}{ | c| c | c |c |c |c |c |}     
        \hline
        \# unknowns & $|v-v_h|_{2,h}$ & Order  & $|v-v_h|_{1,h}$ & Order & $\|v-v_h\|_{L^2}$ & Order\\ 
        \hline \hline
    
        17   &  24.759835  & -        & 4.932699E-0 & -        & 2.069151E-0 & -      \\ 
      \hline
        81   &  15.293270  & 0.6951   & 1.779132E-0 & 1.4712   & 0.727981E-0 & 1.5070 \\ 
      \hline
       353   &  7.8509322  & 0.9619   & 0.483823E-0 & 1.8786   & 0.199644E-0 & 1.8664 \\ 
      \hline
      1473   &  4.0531269  & 0.9538   & 0.137278E-0 & 1.8173   & 0.557622E-1 & 1.8400\\ 
      \hline
      6017   &  2.1219988  & 0.9336   & 0.439086E-1 & 1.6445   & 0.165699E-1 & 1.7507 \\ 
      \hline
     24321   &  1.1421938  & 0.8936   & 0.165883E-1 & 1.4043   & 0.545066E-2 & 1.6040\\
      \hline        
        
       \end{tabular}
       \end{center}
       \caption{Errors and the experimental convergence rates for $v_h$ in broken $H^2, H^1$ and $L^2$ norms for 
L-shaped domain}
       \label{table:OC_Lshape_vh}
       \end{table}

\section{Conclusions \& Perspectives}
\label{conclusions}
In this work, an attempt has been made to obtain approximate solutions for the clamped von K\'{a}rm\'{a}n equations defined on polygonal domains using nonconforming Morley elements. Error estimates in broken energy and $H^1$ norms are established for sufficiently small discretization parameters. Numerical results that substantiate the theoretical results are obtained. A future area of interest would be derivation of reliable $a~posteriori$ error estimates that drive the adaptive mesh refinements.

\medskip

{\bf Acknowledgments: }
The authors would like to sincerely thank Professors S. C. Brenner and Li-yeng Sung for their suggestions on extension of the results to non-convex polygonal domains and  to Dr. Thirupathi Gudi for his comments. The first author would also like to thank National Board for Higher Mathematics (NBHM) for the financial support towards the research work.\\
\medskip

\bibliographystyle{amsplain}
\bibliography{vKeBib}

\section{Appendix}
We consider one of the variants of von K\'{a}rm\'{a}n equations which is important in practical applications and give a brief sketch of the extension of the analysis. 
Consider the following form of von K\'{a}rm\'{a}n equations:
\begin{equation}\label{vkes}
   \left.
    \begin{array}{l l}
      \Delta^2 u &=[u,v]-\frac{p}{D}\Delta u+f \\
      \Delta^2 v &=-\half[u,u]
    \end{array}
   \right\} \text{in } \Omega
  \end{equation}
  with clamped boundary conditions
  \begin{equation}\label{vkbs}
   u=\frac{\partial u}{\partial \nu} = v = \frac{\partial v}{\partial \nu} = 0 \text{  on  } \partial\Omega,
  \end{equation}
where $p$ is a real parameter known as the bifurcation parameter and $D$ denotes the flexural rigidity of the plate. 
The weak formulation of \eqref{vkes}-\eqref{vkbs} reads as: given $F=(f,0)$, find $\Psi\in\cV$ such that
\begin{equation}\label{vforms}
A(\Psi,\Phi)+B(\Psi,\Psi,\Phi)+\fC(\Psi,\Phi)=L(\Phi) \fl \Phi \in  \mathcal{V}
\end{equation}
where $A(\cdot,\cdot),\;B(\cdot,\cdot,\cdot),\; L(\cdot)$ are defined in \eqref{defnA}-\eqref{defnL} respectively, and $\fC(\cdot,\cdot)$ is defined as  
\begin{equation}
\fC(\Theta,\Phi)=-\frac{p}{D}\int_{\Omega}\nabla \theta_1\cdot\nabla\varphi_1\dx\fl\Theta=(\theta_1,\theta_2) \text{ and } \Phi=(\varphi_1,\varphi_2)\in \cV.
\end{equation}
The corresponding nonconforming finite element formulation is given by: find $\Psi_h\in\cV_h$ such that
\begin{equation}\label{vformsd}
A_h(\Psi_h,\Phi)+B_h(\Psi_h,\Psi_h,\Phi)+\fC_h(\Psi_h,\Phi)=L_h(\Phi) \fl \Phi \in  \mathcal{V}_h
\end{equation}
where $A_h(\cdot,\cdot),\;B_h(\cdot,\cdot,\cdot),\; L_h(\cdot)$ are defined in \eqref{defnAh}-\eqref{defnLh} respectively, and $\fC_h(\cdot,\cdot)$ is defined as
\begin{equation}
\fC_h(\Theta,\Phi)=-\frac{p}{D}\sum_{T\in\cT_h}\int_{T}\nabla \theta_1\cdot\nabla\varphi_1\dx\fl\Theta=(\theta_1,\theta_2) \text{ and } \Phi=(\varphi_1,\varphi_2)\in \cV_h.
\end{equation}
For the newly introduced bilinear form $\fC(\cdot,\cdot)$, the following boundedness properties hold true:
\begin{align}
\fC(\Theta,\Phi)&\leq C\trinl\Theta\trinr_1\trinl\Phi\trinr_1\fl\Theta,\Phi\in\cV\label{Cbcont}\\
\fC_h(\Theta_h,\Phi_h)&\leq C\trinl\Theta_h\trinr_{1,h}\trinl\Phi_h\trinr_{1,h}\fl\Theta_h,\Phi_h\in\cV_h\label{Cbnconf}.
\end{align}
For the modified problem \eqref{vforms}, the linearized problem (see \eqref{vformld}) is defined by: for given $G\in\ltsq$, find $\Theta\in\cV$ such that
\begin{equation}\label{vformsl}
     \cA(\Theta,\Phi)=(G,\Phi) \fl \Phi \in \mathcal{V}
   \end{equation}
where
\begin{equation}\label{defnCh}
\cA(\Theta,\Phi):=A(\Theta,\Phi)+B(\Psi,\Theta,\Phi)+B(\Theta,\Psi,\Phi)+\fC(\Theta,\Psi).
\end{equation}
\smallskip
The dual problem is stated as: given $Q\in (H^{-1}(\Omega))^2$, find $\boldsymbol{\zeta}\in\cV$ such that
\begin{equation}\label{vformsldual}
     \cA(\Phi,\boldsymbol{\zeta})=(Q,\Phi) \fl \Phi \in \mathcal{V}.
\end{equation}
It can be observed that if $\Psi$ is an isolated solution of \eqref{vforms}, then \eqref{vformsl} and \eqref{vformsldual} are well posed and satisfy the $a~priori$ bounds
\begin{equation}\label{apriorislin23}
\trinl\boldsymbol{\Theta}\trinr_2\leq C\trinl G\trinr,\:\: \trinl\boldsymbol{\Theta}\trinr_{2+\alpha}\leq C\trinl G\trinr\text{  and  } \trinl\boldsymbol{\zeta}\trinr_2\leq C\trinl Q\trinr_{-1}, \:\: \trinl\boldsymbol{\zeta}\trinr_{2+\alpha}\leq C\trinl Q\trinr_{-1},
\end{equation}
where $\alpha$ is the index of elliptic regularity.
The discrete linearized problem is defined as: find $\Theta_h\in\cV_h$ such that
 \begin{equation}\label{vformsld}
      \cA_h(\Theta_h,\Phi)=(G,\Phi) \fl \Phi \in \mathcal{V}_h
   \end{equation}
where 
\begin{equation}\label{defnAhs}
\cA_h(\Theta_h,\Phi):=A_h(\Theta_h,\Phi)+B_h(\Psi,\Theta_h,\Phi)+B_h(\Theta_h,\Psi,\Phi)+\fC_h(\Theta_h,\Phi).
\end{equation} 

With this background, Theorem~\ref{wellposeld}, Lemma~\ref{nonsingular} and Theorems~\ref{mapball2ball}-\ref{NewtonThm} can be modified for the new formulation, leading to the applicability of the analysis to a more  general form of the von K\'{a}rm\'{a}n equations. We will sketch the proofs of the important results. 

\begin{thm}(Well posedness of discrete linearized problem)\label{wellposesld}
If $\Psi$ is an isolated solution of \eqref{vforms}, then for sufficiently small $h$, the discrete linearized problem \eqref{vformsld} is well-posed.
\end{thm}
\noindent{\it Outline of the proof.} Following the proof of Theorem~\ref{wellposeld}, we easily arrive at \eqref{Thetabdd} using \eqref{Cbnconf}. To estimate $\trinl E_h\Theta_h\trinr_1$ in this case, choose $Q=-\Delta E_h\Theta_h$ and $\Phi=E_h\Theta_h$ in \eqref{vformsldual} and use \eqref{vformsld} to obtain
\begin{align*}
\trinl E_h\Theta_h\trinr_{1}^2
&=A_h(E_h\Theta_h-\Theta_h,\boldsymbol{\zeta})+A_h(\Theta_h,\boldsymbol{\zeta}-\Pi_h\boldsymbol{\zeta})+B_h(\Psi,E_h\Theta_h,\boldsymbol{\zeta}-\Pi_h\boldsymbol{\zeta})+B_h(E_h\Theta_h,\Psi,\boldsymbol{\zeta}-\Pi_h\boldsymbol{\zeta})\\
&\quad+B_h(\Psi,E_h\Theta_h-\Theta_h,\Pi_h\boldsymbol{\zeta})+B_h(E_h\Theta_h-\Theta_h,\Psi,\Pi_h\boldsymbol{\zeta})+(G,\Pi_h\boldsymbol{\zeta})\\
&\quad+\left(\fC_h(E_h\Theta-\Theta_h,\boldsymbol{\zeta})+\fC_h(\Theta_h,\boldsymbol{\zeta}-\Pi_h\boldsymbol{\zeta})\right).
\end{align*}
The last term can be estimated using \eqref{Cbnconf}, Lemmas~\ref{interpolant} and \ref{enrich1} as
\begin{equation}
|\fC_h(E_h\Theta-\Theta_h,\boldsymbol{\zeta})+\fC_h(\Theta_h,\boldsymbol{\zeta}-\Pi_h\boldsymbol{\zeta})|\leq Ch\trinl\boldsymbol{\zeta}\trinr_2.
\end{equation}
The remaining terms are estimated as in Theorem~\ref{wellposeld} and result follows.\qed

\medskip
The next lemma follows as in Lemma~\ref{nonsingular} using \eqref{Cbnconf} and hence the proof is skipped.

\begin{lem}{(Nonsingularity of perturbed bilinear form)}\label{nonsingulars} Let
$\Pi_h\Psi$ be the interpolation of $\Psi$ as defined in Lemma~\ref{interpolant}.
Then, for sufficiently small $h$, the perturbed bilinear form defined by  
\begin{equation}
\tilde \cA_h(\Theta,\Phi)= A_h(\Theta,\Phi)+B_h(\Pi_h\Psi,\Theta,\Phi)+B_h(\Theta,\Pi_h\Psi,\Phi)+\fC(\Theta,\Phi)
\end{equation}
is nonsingular on $\cV_h\times\cV_h$, if \eqref{defnAhs} is nonsingular on $\cV_h\times\cV_h$. 
\end{lem}

\begin{thm}(Mapping of ball to ball) \label{mapball2balls}
For a sufficiently small choice of $h$, there exists a positive constant $R(h)$ such that for any $\Theta \in \cV_h$,
\begin{equation*}
\trinl \Theta-\Pi_h\Psi\trinr_{2,h}\leq R(h)\Rightarrow \trinl\mu(\Theta)-\Pi_h\Psi\trinr_{2,h}\leq R(h).
\end{equation*} 
That is, $\mu$ maps the ball $\bB_{R(h)}(\Pi_h\Psi)$ to itself.
\end{thm}
\noindent{\it Outline of the proof.} Proceeding as in the proof of Theorem~\ref{mapball2ball}, using nonsingularity of $\tilde{\cA}_h(\cdot,\cdot)$ and Lemma~\ref{nonsingulars}, there exists $\bar\Phi\in\cV_h$ such that $\trinl\bar\Phi\trinr_{2,h}=1$ and 
\begin{align}
&\frac{\beta}{4}\trinl\mu(\Theta)-\Pi_h\Psi\trinr_{2,h}\leq \tilde\cA_h(\mu(\Theta)-\Pi_h\Psi, \bphi)\notag\\
&=L_h(\bphi-E_h\bphi) +\left(A_h(\Psi, E_h\bphi)-A_h(\Pi_h\Psi,\bphi)\right)+\left(B_h(\Psi,\Psi, E_h\bphi)-B_h(\Pi_h\Psi,\Pi_h\Psi,\bphi)\right)\notag\\
&\quad+B_h(\Pi_h\Psi-\Theta,\Theta-\Pi_h\Psi,\bphi)\notag+\left(\fC_h(\Psi,E_h\bar\Phi)-\fC_h(\Pi_h\Psi,\bar\Phi)\right)=:\sum_{i=1}^{5} T_i.
\end{align}
The terms $T_1$ to $T_4$ can be estimated as in the proof of Theorem~\ref{mapball2ball}. The last term $T_5$ is estimated using \eqref{Cbnconf}, Lemmas~\ref{interpolant} and \ref{enrich1} as:
\begin{equation}
|\fC_h(\Psi,E_h\bar\Phi)-\fC_h(\Pi_h\Psi,\bar\Phi)|\leq |\fC_h(\Psi,E_h\bar\Phi-\bar\Phi)|+|\fC_h(\Psi-\Pi_h\Psi,\bar\Phi)|\leq Ch\trinl\Psi\trinr_2.
\end{equation}
The remaining proof follows exactly same as the proof of Theorem~\ref{mapball2ball}.\qed\\

\medskip
The existence of solution $\Psi_h$ of \eqref{vformsd} follows using Theorem~\ref{mapball2balls} and satisfies the estimate
\begin{equation}\label{solnee}
\trinl\Psi_h-\Pi_h\Psi\trinr_{2,h}\leq Ch^\alpha.
\end{equation}
A contraction result similar to Theorem~\ref{contractionthm} also holds true in this case. The energy estimate follows exactly as in the proof of Theorem~\ref{eetimate}. 

\begin{thm}{($H^1$ estimate)}\label{h1estimate_s}
Let $\Psi$ and $\Psi_h$ be the solutions of \eqref{vforms} and \eqref{vformsd} respectively. Assume that $\Psi$ is an isolated solution. Then, for sufficiently small $h$, it holds
\begin{equation}
\trinl\Psi-\Psi_h\trinr_{1,h}\leq C h^{2\alpha},
\end{equation}
where $\alpha\in (\half,1]$ is the index of elliptic regularity.
\end{thm}
\noindent{\it Outline of the proof.} A use of triangle inequality yields
\begin{equation}\label{trienqs}
\trinl\Psi-\Psi_h\trinr_{1,h}\leq \trinl\Psi-\Pi_h\Psi\trinr_{1,h}+\trinl\Pi_h\Psi-\Psi_h\trinr_{1,h}\leq \trinl\Psi-\Pi_h\Psi \trinr_{1,h}+\trinl\boldsymbol{\rho}-E_h\boldsymbol{\rho}\trinr_{1,h}+\trinl E_h\boldsymbol{\rho}\trinr_{1},
\end{equation}
 where $\boldsymbol{\rho}=\Pi_h\Psi-\Psi_h$. A choice of $Q=-\Delta E_h\boldsymbol{\rho}$ and $\Phi=E_h\boldsymbol{\rho}$ in the dual problem~\eqref{vformsldual} leads to
\begin{align}
\left(\nabla E_h\boldsymbol{\rho},\nabla E_h\boldsymbol{\rho}\right) &= \cA_h({E_h\boldsymbol{\rho},\boldsymbol \zeta})=\cA_h(E_h\boldsymbol{\rho}-\boldsymbol{\rho},{\boldsymbol \zeta})+\cA_h(\boldsymbol{\rho},{\boldsymbol \zeta})\nonumber\\
&= A_h(E_h\boldsymbol{\rho}-\boldsymbol{\rho},{\boldsymbol \zeta})+B_h(\Psi,E_h\boldsymbol{\rho}-\boldsymbol{\rho},{\boldsymbol \zeta})+B_h(E_h\boldsymbol{\rho}-\boldsymbol{\rho},\Psi,{\boldsymbol \zeta})\nonumber+\fC_h(E_h\boldsymbol{\rho}-\boldsymbol{\rho},{\boldsymbol \zeta})\\
&\quad+A_h(\Pi_h\Psi-\Psi,{\boldsymbol\zeta})+A_h(\Psi-\Psi_h,\boldsymbol{\zeta}-\Pi_h\boldsymbol{\zeta})+A_h(\Psi,\Pi_h\boldsymbol{\zeta}-\boldsymbol{\zeta})+L_h(\boldsymbol{\zeta}-\Pi_h\boldsymbol{\zeta})\nonumber\\
&\quad
+\left(B_h(\Psi,\Pi_h\Psi-\Psi_h,{\boldsymbol \zeta})+B_h(\Pi_h\Psi-\Psi_h,\Psi,{\boldsymbol \zeta})-B_h(\Psi,\Psi,{\boldsymbol \zeta})+B_h(\Psi_h,\Psi_h,\Pi_h{\boldsymbol \zeta})\right)\nonumber\\
&\quad+
 \left(\fC_h(\Pi_h\Psi-\Psi,\boldsymbol{\zeta})+\fC_h(\Psi_h,\Pi_h\boldsymbol{\zeta}-\boldsymbol{\zeta})\right)
\end{align}
Combining all the terms related to $\fC_h$ and using \eqref{Cbnconf}, \eqref{solnee} and Lemmas~\ref{interpolant}, \ref{enrich1}, we obtain the estimate
\begin{equation}
\fC_h(E_h\boldsymbol{\rho}-\boldsymbol{\rho},{\boldsymbol \zeta})+\fC_h(\Pi_h\Psi-\Psi,\boldsymbol{\zeta})+\fC_h(\Psi_h,\Pi_h\boldsymbol{\zeta}-\boldsymbol{\zeta})\leq Ch^{1+\alpha}\trinl\boldsymbol{\zeta}\trinr_{2+\alpha}.
\end{equation}
Estimating the remaining terms as in the proof of Theorem~\ref{h1estimate}, the result follows.\qed\\

\medskip
The Newton's iterates in this case are defined by
\begin{equation}\label{NewtonIterate_s}
A_h(\Psi_h^{n},\Phi)+B_h(\Psi_h^{n-1},\Psi_h^{n},\Phi)+B_h(\Psi_h^{n},\Psi_h^{n-1},\Phi)+\fC_h(\Psi_h^n,\Phi)=B_h(\Psi_h^{n-1},\Psi_h^{n-1},\Phi)+L_h(\Phi)\fl \Phi\in \cV_h.
\end{equation}
The quadratic convergence result follows by a similar proof as in Theorem~\ref{NewtonThm}.\\

\medskip
\subsection{Example 3}\label{example3}
In this example, we perform numerical experiments for the problem \eqref{vkes}-\eqref{vkbs} with $p/D=10$, over a unit square domain. Choose the right hand side load functions such that the exact solution is given by
 \begin{align*}
     u(x,y)&=x^2(1-x)^2y^2(1-y)^2,\quad v(x,y)=\sin^2(\pi x)\sin^2(\pi y).
 \end{align*}
We consider the same initial triangulation and its uniform refinement process as in Example~\ref{example1}. Tables ~\ref{table:OC_NCFEM_uhs} and ~\ref{table:OC_NCFEM_vhs} show the errors and experimental convergence rates for the variables $u_h$ and $v_h$. The computational order of convergences in broken $H^2,\; H^1$ norms are quasi-optimal and verify the theoretical results. Also, the order of convergence with respect to $L^2$ norm is sub-optimal justifying the results in~\cite{HuShi}.

 \begin{table}
   \begin{center}
     \begin{tabular}{ | c| c | c |c |c |c |c |}     
      \hline
      \# unknowns & $|u-u_h|_{2,h}$ & Order  & $|u-u_h|_{1,h}$ & Order & $\|u-u_h\|_{L^2}$ & Order\\ 
       \hline \hline
      
     25    &  0.101724E-0  & -         & 0.129574E-1 & -        & 0.469669E-2  & -\\ 
       \hline
     113   &  0.391714E-1  & 1.3767    & 0.275863E-2 & 2.2317   & 0.957470E-3  & 2.2943\\ 
       \hline
     481   &  0.195023E-1  & 1.0061    & 0.767382E-3 & 1.8459   & 0.252196E-3  & 1.9246\\ 
       \hline
     1985  &  0.974844E-2  & 1.0004    & 0.198544E-3 & 1.9504   & 0.641987E-4  & 1.9739\\ 
       \hline
     8065  &  0.487399E-2  & 1.0000    & 0.500990E-4 & 1.9866   & 0.161298E-4  & 1.9928\\
       \hline 
     32513 &  0.243697E-2  & 1.0000    & 0.125546E-4 & 1.9965   & 0.403763E-5  & 1.9981\\
       \hline
     \end{tabular}
    \end{center}
     \caption{Errors and Convergence rates of $u_h$  in broken $H^2,H^1$ and $L^2$ norms}
     \label{table:OC_NCFEM_uhs}
 \end{table}
 
\begin{table}
    \begin{center}
       \begin{tabular}{ |c| c | c |c |c |c |c |}     
         \hline
         \# unknowns & $|v-v_h|_{2,h}$ & Order  & $|v-v_h|_{1,h}$ & Order & $\|v-v_h\|_{L^2}$ & Order\\ 
          \hline \hline
 
        25    & 19.245650 & -        & 2.140609E-0 &  -       & 0.770875E-0 &  - \\ 
        \hline
        113   & 9.5043692 & 1.0178   & 0.569978E-0 & 1.9090   & 0.177898E-0 & 2.1154 \\ 
          \hline
        481   & 5.0549208 & 0.9109   & 0.161737E-0 & 1.8172   & 0.482777E-1 & 1.8816 \\ 
          \hline
        1985  & 2.5758938 & 0.9726   & 0.421546E-1 & 1.9398   & 0.123930E-1 & 1.9618 \\ 
          \hline
        8065  & 1.2944929 & 0.9926   & 0.106618E-1 & 1.9832   & 0.312076E-2 & 1.9895\\
          \hline  
        32513 & 0.6480848 & 0.9981   & 0.267351E-2 & 1.9956   & 0.781643E-3 & 1.9973\\
          \hline
       \end{tabular}
    \end{center}
    \caption{Errors and Convergence rates of $v_h$  in broken $H^2,H^1$ and $L^2$ norms}\label{table:OC_NCFEM_vhs}
 \end{table}
    
\end{document}